\documentclass[10pt,a4paper]{amsart}
\usepackage{amsmath,amssymb, amsbsy}
\usepackage{color,psfrag}
\usepackage[dvips]{graphicx}
\usepackage{color}
\theoremstyle{plain}
\newtheorem{thm}{Theorem}[section]
\theoremstyle{plain}
\newtheorem{lem}[thm]{Lemma}
\newtheorem{prop}[thm]{Proposition}
\newtheorem{cor}[thm]{Corollary}
\newtheorem{ex}{Example}[section]
\theoremstyle{definition}

\newtheorem{rem}{Remark}[section]

\newcommand{\rn}{\mathbb{R}^{N}}
\newcommand{\bn}{\mathbb{B}^{N}}

\newcommand{\hn}{\mathbb{H}^{N}}

\newcommand{\authorfootnotes}{\renewcommand\thefootnote{\@fnsymbol\c@footnote}}%

\numberwithin{equation}{section} \allowdisplaybreaks

\usepackage[text={6in,8.6in},centering]{geometry}
\begin{document}
        \title[]{Sharp Poincar\'e-Hardy\\ and Poincar\'e-Rellich inequalities\\ on the hyperbolic space}

\date{}

\author[Elvise Berchio]{Elvise Berchio}
\address{\hbox{\parbox{5.7in}{\medskip\noindent{Dipartimento di Scienze Matematiche, \\
Politecnico di Torino,\\
        Corso Duca degli Abruzzi 24, 10129 Torino, Italy. \\[3pt]
        \em{E-mail address: }{\tt elvise.berchio@polito.it}}}}}
\author[Debdip Ganguly]{Debdip Ganguly}
\address{\hbox{\parbox{5.7in}{\medskip\noindent{Dipartimento di Scienze Matematiche, \\
Politecnico di Torino,\\
        Corso Duca degli Abruzzi 24, 10129 Torino, Italy. \\[3pt]
        \em{E-mail address: }{\tt debdip.ganguly@polito.it}}}}}
\author[Gabriele Grillo]{Gabriele Grillo}
\address{\hbox{\parbox{5.7in}{\medskip\noindent{Dipartimento di Matematica,\\
Politecnico di Milano,\\
   Piazza Leonardo da Vinci 32, 20133 Milano, Italy. \\[3pt]
        \em{E-mail addresses: }{\tt
          gabriele.grillo@polimi.it}}}}}




\begin{abstract} We study Hardy-type inequalities associated to the quadratic form of the shifted Laplacian $-\Delta_{\mathbb H^N}-(N-1)^2/4$ on the hyperbolic space ${\mathbb H}^N$, $(N-1)^2/4$ being, as it is well-known, the bottom of the $L^2$-spectrum of $-\Delta_{\mathbb H^N}$. We find the optimal constant in a resulting Poincar\'e-Hardy inequality, which includes a further remainder term which makes it sharp also locally: the resulting operator is in fact critical in the sense of \cite{pinch}. A related improved Hardy inequality on more general manifolds, under suitable curvature assumption and allowing for the curvature to be possibly unbounded below, is also shown. It involves an explicit, curvature dependent and typically unbounded potential, and is again optimal in a suitable sense. Furthermore, with a different approach, we prove Rellich-type inequalities associated with the shifted Laplacian, which are again sharp in suitable senses.
\end{abstract}

\maketitle

\tableofcontents

 \section{Introduction}
 The problem of existence of optimal, namely ``as large as possible'', Hardy weights dates back to \cite{A} and has been brought to a high level of sofistication, see e.g., and without any claim of completeness the papers \cite{GFT, BFT2,BT, Mitidieri1, Mitidieri2, BrezisM, Dambrosio, FT, gaz, Kombe1, MMP, Mitidieri, PT1, PT2} and references quoted therein. By a Hardy weight we mean a non zero nonnegative function $W$ such that the following inequality
 \begin{equation}\label{1}
 q(u)\ge \int_\Omega Wu^2\,{\rm d}x,\ \ \ \forall u\in C_c^\infty(\Omega),
 \end{equation}
 holds true, where $\Omega$ is a (e.g. Euclidean) domain and $q(u)=(u,Pu)$ is the quadratic form of a linear, elliptic, second order, symmetric, non-negative operator $P$ on $\Omega$.

 In several of the above mentioned papers, \it improved \rm versions of classical Hardy inequalities are dealt with, starting from the seminal papers by Brezis and Vazquez \cite{Brezis} and Brezis and Marcus \cite{BrezisM}. The recent paper by Devyver, Fraas and Pinchover (\cite{pinch}) deals with general second order subcritical elliptic operators $P$, either on domains in $\rn$ or on noncompact manifolds, and provides optimal weights in Hardy-type inequalities related to the quadratic form of $P$, in terms of properties of positive supersolutions of $Pu=0$.

 As concerns the analogue of the classical Euclidean Hardy inequality on Riemannian manifolds, G. Carron \cite{Carron} has shown that the inequality
 \begin{equation}\label{Euclidean}
 \int_M|\nabla_g u|^2\,{\rm d}v_g\ge \frac{(N-2)^2}4\int_M\frac{u^2}{\varrho(x,o)^2}\,{\rm d}v_g\ \ \ \forall u\in C_c^\infty(M)
 \end{equation}
 holds on any Cartan-Hadamard manifold (namely a manifold which is complete, simply-connected, and has everywhere non-positive sectional curvature),  $\varrho$ denoting geodesic distance, whereas $\nabla_g, {\rm d}v_g$ now indicate the Riemannian gradient and measure. Notice that the constant $(N-2)^2/4$ coincides with its optimal Euclidean counterpart. Further results are given in the recent papers \cite{Dambrosio,Kombe1, Yang}.

 On the other hand Cartan-Hadamard manifolds whose sectional curvatures are bounded above by a \it strictly negative \rm constant, are known to admit a Poincar\'e type, or $L^2$-gap, inequality, namely there exists $\Lambda>0$ such that
 \begin{equation*}
 \int_M|\nabla_g u|^2\,{\rm d}v_g\ge \Lambda\int_Mu^2\,{\rm d}v_g\ \ \ \forall u\in C_c^\infty(M).
 \end{equation*}
The most classic example one has in mind is of course the \it hyperbolic space \rm ${\mathbb H}^N$, where $\Lambda=(N-1)^2/4$. Furthermore, it is known that the $L^2$-spectrum of the Riemannian Laplacian is the half line $[\Lambda, \infty)$ and that the infimum
 \begin{equation}\label{poin}
\Lambda:=\lambda_1(\hn):= \inf_{u \in H^{1}(\hn) \setminus \{ 0\}} \frac{\int_{\hn}
|\nabla_{\hn} u|^2 \ dv_{\hn}}{\int_{\hn} |u|^2 \ dv_{\hn}}
\end{equation}
is never achieved.

Our first goal here will be to deal with \it sharp, improved \rm Hardy inequalities on the hyperbolic space, where we take the attitude that the improvement is done on the \it gap, or Poincar\'e \rm inequality \eqref{poin}, and in particular we are interested in the following:

\vspace{.3truecm}
\noindent \bf Problem 1. \it Does there exist $c>0$ such that the following Poincar\'e-Hardy inequality
\begin{equation}\label{question}
\int_{{\mathbb H}^N}|\nabla_{\hn} u|^2\,{\rm d}v_{\hn}-\frac{(N-1)^2}4\int_{{\mathbb H}^N}u^2\,{\rm d}v_{\hn}\ge c\int_{{\mathbb H}^N}\frac{u^2}{r^2}\,{\rm d}v_{\hn}\ \ \ \forall u\in C_c^\infty(\hn)
\end{equation}
holds, where $r:=\varrho(x,o)$ and $o\in{\mathbb H}^N$ is fixed? Which is the optimal value of $c$ if such a constant exists? Is the resulting inequality further improvable to yield criticality of a suitable Schr\"odinger operator? Does any improved Hardy inequality hold on more general manifolds under curvature conditions, and if yes is it sharp in a suitable sense?
\vspace{.3truecm}

\rm It is clear that, if the above problem has a positive answer, the constant $(N-1)^2/4$ in the l.h.s. of \eqref{question} is sharp by construction. It is also clear that \eqref{question} has no Euclidean counterpart, in contrast with \eqref{Euclidean}.

One should notice that Problem 1 is different from that treated in \cite{Kombe1, Kombe2, Yang}, where the optimal Hardy constant $(N-2)^2/4$ is taken as \it fixed\rm, and one looks for bounds for the constant in front of $\|u\|_{L^2}^2$, or for some different reminder terms. Such approach resembles instead more closely the kind of improvements given in the case of Euclidean bounded domains by \cite{BrezisM, Brezis}, a setting in which the value of the optimal Poincar\'e constant is in general not known.

In regard to Problem 1, we notice that a positive answer to its first question is suggested, on the one hand, by the explicit bounds for the heat kernel on ${{\mathbb H}^N}$ (see e.g. \cite{DA}) which show that the nonnegative operator $-\Delta_{\mathbb H^N}-(N-1)^2/4$ admits a Green's function (for $N\ge3$), and hence an inequality like \eqref{question}, with the weight $r^{-2}$ replaced by a suitable positive weight $W$, holds. On the other hand, the supersolution construction of \cite{pinch}, using as ingredients the known asymptotic behavior of the Green's function of the shifted Laplacian $P=-\Delta_{\mathbb H^N}-\Lambda$ and of the positive radial solution of the equation $Pu=0$, yields, after an easy calculation which is omitted here, that the decay at infinity of the corresponding optimal Hardy weight should be exactly $c\,r^{-2}$ for a suitable $c>0$. It is important to remark that this method \it does not give a sharp value for $c$ \rm since some of the quantities involved are not known explicitly with the detail needed.  A similar phenomenon occurs in the Euclidean situation when dealing directly with the shifted operator $-\Delta+1$. Of course, \it a posteriori \rm one could reformulate the supersolution construction given in Section 4 in terms of the shifted Laplacian.

In Theorem \ref{poincare} below, we shall answer in more detail this question by proving (an improvement of) the following inequality, which relies on a supersolution technique:

 \begin{equation}\label{poincareeq1}
\int_{\hn} |\nabla_{\hn} u|^2 \ dv_{\hn} - \frac{(N-1)^2}{4} \int_{\hn} u^2 \ dv_{\hn}
\geq \frac{1}{4} \int_{\hn} \frac{u^2}{r^2} \ dv_{\hn}\,
\end{equation}
for all $ u \in C^{\infty}_{c}(\hn).$ Furthermore, the constant $\frac{1}{4}$ in \eqref{poincareeq1} is sharp. In fact, we shall prove a stronger inequality, involving an additional positive remainder term, call it $w$, with a second \it optimal \rm constant, which tends to reproduce better and better the Euclidean Hardy inequality, with optimal constant, for functions with support in a Riemannian ball $B_\varepsilon(o)$, as $\varepsilon\to 0$. Notice that our result entails that the operator
\[
P_{1}:=-\Delta_{\mathbb H^N}-\frac{(N-1)^2}4-\frac1{4r^2}
\]
beside being nonnegative is also \it subcritical\rm, hence in particular it admits a positive minimal Green's function, and this is not true if the constant $1/4$ is replaced by any larger one. Furthermore, the operator
\[
P_{2}:=-\Delta_{\mathbb H^N}-\frac{(N-1)^2}4-\frac1{4r^2}-w,
\]
$w$ being the additional positive remainder term mentioned above, is critical in the sense of \cite[Definition 2.1]{pinch} hence no further positive weight may be added to the r.h.s. of the quadratic form inequality we prove, see Remark \ref{remarkoptimal}.

Clearly, when restricted to functions supported on a fixed geodesic ball, $P_2$ is no more critical and in Proposition \ref{infinity} we provide, as a sample of further generalization of the previous methods, an infinite expansion of logarithmic weights that can be added to the r.h.s., with sharp constants.

 After completing this paper, we got aware of the paper \cite{AK}, where inequality \eqref{poincareeq} is proved in $\hn$, but with a different proof. Also the optimality issues, which are our main task here, are addressed there in a different and less direct way, involving spectral properties of Schr\"odinger operators, and not dealing with criticality issues.  Indeed, our methods exploit the explicit knowledge of radial solutions suitable combined with the criticality theory developed in \cite{pinch}. Furthermore, the arguments applied are flexible enough to allow to prove sharp inequalities, and in a natural way criticality for related Schr\"odinger operators, also on more general manifolds under upper curvature bounds, see Theorem \ref{general manifold}. Improved Hardy type inequalities are also shown to hold in more general manifolds in \cite{AK}, but they are not stated in terms of (upper) curvature bounds.

We are aware of few Hardy-type inequalities which are related with ours. A first one can be deduced as an application of \cite[Theorem 2.2]{pinch}, by which an optimal weight for the laplacian in $\hn\setminus \{o\}$ is $\frac1{4}\left(\frac{G'(r)}{G(r)}\right)^2$ where, for a suitable positive constant $c$, $G(r)=c\,\int_r^{+\infty} (\sinh s)^{-(N-1)} \,ds$ is the Green function of $-\Delta_{\mathbb H^N}$. Since $\frac1{4}\left(\frac{G'(r)}{G(r)}\right)^2\geq \Lambda$ for every $r>0$, the corresponding inequality \eqref{1} can be read as an improvement of \eqref{poin}. The above weight behaves like the Hardy weight \eqref{Euclidean} near $0$ but converges to $\Lambda$ exponentially fast at infinity, hence it does not give an answer to Problem 1. It's worth noting that in \cite[Example 5.3]{BMR} it is shown how the weight $\frac1{4}\left(\frac{G'(r)}{G(r)}\right)^2$ can be computed by an iterative argument. One sees that the resulting weight is larger than $1/4r^2$ when $r$ is small and to be  smaller than $1/4r^2$ when $r$ is large,
 so the two inequalities are not comparable, as expected since both weights are optimal. The above argument works for model manifolds also, by exploiting the corresponding (known) Green's function, which provides however a much less explicit weight, involving an integral function, when compared to the result given below in Theorem \ref{general manifold}. A second inequality bearing some resemblance with ours is proved in \cite[Example 1.8]{LW}, where a Hardy-type inequality in terms
of a weight weight $w(r)$ tending to $\Lambda$ as $r\to+\infty$, but behaving as $\textrm{const}/r$ as $r\to0$, is shown on general Cartan-Hadamard manifolds with
sec\,$\le-1$. 

When $N=3$, $\frac{1}{4}$
is exactly the classical Hardy constant $\frac{(N-2)^2}{4}$ and \eqref{poincareeq1} can also be seen as an optimal
 Hardy inequality with an optimal $L^2$ remainder term. 
 See also Remark \ref{rem1}. \par
 It is worth noting that, after performing a suitable ``conformal change of metric'', \eqref{poincareeq1} yields an Hardy inequality in the Euclidean ball involving the distance from the boundary, see Corollary \ref{cor}, which is a slight improvement upon a (already optimal) inequality given in \cite{GFT} and seems not to be known. See \cite{GFT,BrezisM} for further improved Euclidean Hardy inequalities involving the distance from the boundary. In a similar way, in Corollary \ref{corHALF1} we provide a nonstandard remainder term for the Hardy-Maz'ya inequality \cite[2.1.6 Corollary 3]{Ma} in the half-space. See also Corollary \ref{mazya}.
 \par

 \subsection{General Cartan-Hadamard manifolds} By the same strategy used on ${\mathbb H}^N$, one can prove related inequalities on model (i.e. spherically symmetric) manifolds, and this enables us to extend the previous result to general manifolds under appropriate curvature assumptions, which allow for sectional curvatures possibly unbounded below. This is the content of Theorem \ref{general manifold}. While negative curvature always implies that a suitable Hardy inequality holds (see \cite{Carron}) it is conceivable that \it unbounded \rm negative curvature implies that the constant term $(N-1)^2/4$ above can be replaced by an unbounded, nonconstant positive potential. In fact, the Hardy weight we construct is explicitly related to sectional curvature in the model manifold naturally associated to the curvature bounds assumed. The weight is unbounded when sectional curvature is unbounded below, thus in particular giving rise to a Schr\"odinger operator $H=-\Delta-V$ with positive, unbounded potential $V$, which is nevertheless controlled from below by the Hardy potential, so that $H\ge1/4r^2+$ (positive remainder terms). The previous result on $\hn$ is of course a special case of this fact. This is our second main result and we stress that this result is again sharp in the following sense: given any $\psi$ as in Theorem \ref{general manifold} there exist a manifold satisfying the upper bound on curvature as requested in \eqref{condition2} in terms of $\psi$ and such that the Schr\"odinger operator defined in Theorem \ref{general manifold}, and involving the Hardy term, is critical.

 \subsection{Rellich-Poincar\'e inequalities} The final topic we shall deal with here is concerned with the validity of \it Rellich-Poincar\'e inequalities\rm, namely inequalities involving the quadratic form of the shifted operator $\Delta_{\hn}^2-\Lambda^2$, where as above $\Lambda=(N-1)^2/4$. Rellich inequalities in the Euclidean setting go back to \cite{R}, and a number of refinement and improvements have been given till quite recently, see e.g. without any claim for completeness \cite{BT, CM, DH, gaz, GM, Mitidieri,TZ}.  The very recent paper \cite{MSS} proposed a method of proof involving a decomposition in spherical harmonics, which turns out to be useful in the present case as well. See also \cite{TZ} and \cite{VZ} where spherical harmonics were applied in the context of Hardy and Rellich inequalities. The basic Euclidean inequality one starts from is the following well-known one:
\begin{equation*}
 \int_{\rn}|\Delta u|^2\,{\rm d}x\ge\frac{N^2(N-4)^2}{16}\int_{\rn}\frac{u^2}{|x|^4}\,{\rm d}x,
\end{equation*}
 valid for all $u\in C_c^\infty(\rn)$ provided $N\ge5$.

Likewise, various forms of Rellich inequalities on $\hn$, including improved ones, have been proved recently in \cite{Kombe1, Kombe2}. We are not aware of further results in this connection and, also motivated by the fact that the following infimum is never attained
 $$
 \inf_{u \in H^{2}(\hn) \setminus \{ 0\}} \frac{\int_{\mathbb{H}^{N}}
|\Delta_{\hn} u|^2 \ dv_{\hn}}{\int_{\hn} |u|^2 \ dv_{\hn}} = \frac{(N-1)^4}{16}\,,
$$
we shall be interested here to deal with the following analogue for higher order of Problem 1:

\vspace{.3truecm}
\noindent \bf Problem 2. \it Does there exist a nonnegative, non identically zero weight $w$ such that the following Rellich-Poincar\'e inequality
\begin{equation}\label{question2}
\int_{{\mathbb H}^N}|\Delta_{\hn} u|^2\,{\rm d}v_{\hn}-\frac{(N-1)^4}{16}\int_{{\mathbb H}^N}u^2\,{\rm d}v_{\hn}\ge \int_{{\mathbb H}^N}w\,u^2\,{\rm d}v_{\hn}
\end{equation}
holds for all $u\in C_c^\infty(\hn)$?
\vspace{.3truecm}
\rm

It is again clear that, if Problem 2 has a positive answer, the constant $(N-1)^4/16$ in the l.h.s. of \eqref{question2} is sharp by construction.

We shall show in Theorem \ref{TPR} that the answer to Problem 2 is affirmative, and show that one can take, setting as before $r=\varrho(x,o)$:
\[
w(x)=\frac{(N-1)^2}{8r^2} +\frac{9}{16\,r^4} + (\textrm{positive correction terms})
\]
In Section \ref{Rellich}  we show that the constant $\frac{(N-1)^2}{8}$ is sharp and we state some facts pointing towards the optimality of $\frac{9}{16}$. It should be remarked that:
\begin{itemize}
\item The positive correction terms in the above expression of $w$ are such that
\[
w(x)\sim \frac{N^2(N-4)^2}{16\,r^4}\ \ \ {\rm as}\ r\to 0,
\]
where the r.h.s. is exactly the optimal Euclidean weight. In such sense, our bound recovers the Euclidean Rellich inequality for functions supported in a ball with small radius. See Remark \ref{rellichsharpness} for a precise statement;
\item After having remarked that the weight $w$ has the sharp Euclidean behaviour for small $r$, it should be noted that the leading term in $w$ is instead the one involving the quantity $1/r^2$ for functions supported \it outside \rm a large ball, namely as $r\to +\infty$. Hence, it is particularly important to determine the sharp constant in front of such a term to capture the non-Euclidean feature (e.g. the leading term when $r$ is large) of the inequality we prove. Notice that the term of the form $1/r^{2}$, which already appeared in some of the (Euclidean) results of \cite{gaz}, of course does not violate any scale invariance for the inequality we consider. The problem of finding the best constant when $w$ is of the form $c/r^4$ remains however open. See however Remark \ref{conj} for some clue pointing towards sharpness of the constant 9/16 found here.
\end{itemize}

We stress that, although the statements look very similar, the proof of our Poincar\'e-Rellich inequality is completely different from the one of \eqref{poincareeq1}. Here, orthogonal decomposition in spherical harmonics and a suitable 1-dimensional Hardy type inequality are the main tools exploited. As in the first order case, we give a sample of the results which can be derived, in the Euclidean space, from our main result, see Corollary \ref{cor2}. When restricting to radial functions a further Euclidean inequality is derived in Proposition \ref{onedimensional}.
 \par

 The paper is organized as follows: in Section \ref{Hardy} we introduce some of the notations and some geometric definitions and we state our Poincar\'{e}-Hardy inequality first on ${\mathbb H}^N$, and then on more general manifolds under sectional curvature assumptions. When $M$ is the hyperbolic space, we give the precise statement of a refinement of \eqref{poincareeq1} in Theorem \ref{poincare} and of the associated Euclidean inequality in Corollary \ref{cor}. It is worth noticing that the weight appearing in the general Theorem \ref{general manifold} has a precise geometrical meaning in terms of sectional curvature of a model manifold, modeled on a function $\psi$ in terms of which
the relevant curvature assumptions are given.

 In Section \ref{Rellich} we state our Poincar\'{e}-Rellich inequality and some Euclidean Rellich inequalities derived from it in the half space. Sections \ref{proof1} contains the proof of the Poincar\'e-Hardy inequality on ${\mathbb H}^N$ and of Theorem \ref{general manifold}. In Section \ref{hyperbolic proof}  we give and alternative proof of optimality in Theorem \ref{poincare} and prove Corollary \ref{cor} as well. Section \ref{rellich-proof} contains the proof of the Poincar\'e-Rellich inequality while Section 7 contains the proof of Corollary \ref{corHALF1}, Corollary \ref{mazya} and Corollary \ref{cor2}. Finally, the proof of Proposition \ref{infinity} is given is Section \ref{loginfinity}.

  \section{Poincar\'{e}-Hardy inequalities}\label{Hardy}

  We state here our main result about Poincar\'e-Hardy inequalities on ${\mathbb H}^N$. Below, $r:=\varrho(x,o)$ for a given pole $o\in{\mathbb H}^N$ and $B_\varepsilon:=\{x\in\hn, \varrho(x,o)<\varepsilon\}$.

  \begin{thm}\label{poincare}
 Let $ N \geq 3$. For all $ u \in C^{\infty}_{c}(\mathbb{H}^{N})$ there holds
 \begin{equation}\label{poincareeq}\begin{aligned}
\int_{\hn} |\nabla_{\hn} u|^2 \ dv_{\hn} - \frac{(N-1)^2}{4} \int_{\hn} u^2 \ dv_{\hn}
\geq &\frac{1}{4} \int_{\hn} \frac{u^2}{r^2} \ dv_{\hn}\\ &+ \frac{(N-1)(N-3)}{4} \int_{\hn} \frac{u^2}{\sinh^2 r} \ dv_{\hn}.
\end{aligned}\end{equation}
Besides, the operator
\[
H=-\Delta_{\mathbb H^N}-\frac{(N-1)^2}4-\frac1{4\,r^2}-\frac{(N-1)(N-3)}{4}  \frac{1}{\sinh^2 r}
\]
is critical in $\mathbb{H}^{N}\setminus\{o\}$ in the sense of \cite[Definition 2.1]{pinch}; that is, the inequality
$$\int_{\hn} |\nabla_{\hn} u|^2 \ dv_{\hn} - \frac{(N-1)^2}{4} \int_{\hn} u^2 \ dv_{\hn}
\geq \int_{\hn}V u^2 \ dv_{\hn}\quad \forall u\in  C^{\infty}_{c}(\mathbb{H}^{N}\setminus\{o\})$$
is not valid for any $V>\frac{1}{4\,r^2}+ \frac{(N-1)(N-3)}{4\, \sinh^2 r}$. \par The constant $\frac{(N-1)^2}{4}$ in \eqref{poincareeq} is of course sharp in the sense that the l.h.s. of \eqref{poincareeq} can be negative if such constant is replaced by a larger one, and the criticality of the operator $H$ yields that also the constant $\frac{1}{4}$ in \eqref{poincareeq} is sharp in the sense that no inequality of the form
\[
\int_{\hn} |\nabla_{\hn} u|^2 \ dv_{\hn} - \frac{(N-1)^2}{4} \int_{\hn} u^2 \ dv_{\hn}\ge
c\, \int_{\hn} \frac{u^2}{r^2} \ dv_{\hn}
\]
holds for all $ u \in C^{\infty}_{c}(\mathbb{H}^{N})$ when $c>1/4$. Finally, the constant $\frac{(N-1)(N-3)}{4}$ is sharp as well in the sense that no inequality of the form
\[
\int_{\hn} |\nabla_{\hn} u|^2 \ dv_{g} - \frac{(N-1)^2}{4} \int_{\hn} u^2 \ dv_{\hn}
\geq \frac{1}{4} \int_{\hn} \frac{u^2}{r^2} \ dv_{\hn}+ c \int_{\hn} \frac{u^2}{\sinh^2 r} \ dv_{\hn}.
\]
holds, given any $\varepsilon>0$, for all $ u \in C^{\infty}_{c}(B_\varepsilon)$ when $c>(N-1)(N-3)/4$.

\end{thm}


\begin{rem}\label{remarkoptimal}Set
$$P: =  -\Delta_{\mathbb{H}^{N}} -  \frac{(N-1)^2}{4} - \frac{(N-1)(N-3)}{4\sinh^2 r}\quad \text{and} \quad W (r): = \frac{1}{4r^2}\ \,,$$
by Theorem \ref{poincare} the operator $P-W$ is critical and since the corresponding ground state does not lie in  $L^{2}(\mathbb{H}^{N} \setminus \{o\}, W),$ $P-W$ is also null critical, see the proof of Theorem \ref{poincare} and \cite[Definition 4.8]{pinch}. Furthermore, arguing as in \cite[Example 3.1]{pinch}, since for $\eta >1$ the radial solutions of the equation $P v=\eta W\, v$ oscillate near zero and near infinity it follows that the best possible constant for the validity of the inequality associated to $P-\eta W$, in any neighborhood of either the origin or infinity, is $\eta=1$. Besides, the bottom of the spectrum and the bottom of the essential spectrum of $W^{-1} P$ is $1$.
\end{rem}

\begin{rem}\label{rem1}
Recalling \eqref{poin}, \eqref{poincareeq} can be seen as an improvement of the (best possible) Poincar\'{e} inequality where an optimal Hardy remainder terms have been added.  On the other hand, when $N=3$, $\frac{1}{4}$
is exactly the classical Hardy constant $\frac{(N-2)^2}{4}$ and \eqref{poincareeq} can also be seen as the (best possible)
 Hardy inequality with an $L^2$ remainder term.\par Besides, if $ N \geq3$ and $\lambda \in[0, \lambda_1(\hn)]$, from Theorem \ref{poincare}  it is easily deduced the existence of a positive constant $h(\lambda)$ such that the following family of inequalities holds
\begin{equation}\label{lambda}
\int_{\hn} |\nabla_{\hn} u|^2 \ dv_{\hn} -\lambda \int_{\hn} u^2 \ dv_{\hn}
\geq h(\lambda) \int_{\hn} \frac{u^2}{r^2} \ dv_{\hn},
 \end{equation}
 for every $ u \in C^{\infty}_{0}(\hn)$. Moreover one has:
\begin{itemize}
\item[$\bullet$] $h(0)=\frac{(N-2)^2}{4}$ is the Euclidean Hardy constant and equality
 in \eqref{lambda} is not achieved;
\item[$\bullet$] $h(\lambda_1(\hn))=\frac{1}{4}$ and equality in \eqref{lambda} is not achieved;
\item[$\bullet$] the map $\lambda \mapsto h(\lambda)$ is non increasing and concave, hence continuous.
\end{itemize}
Furthermore, from \cite[Theorem 5.2]{Yang} we know that
$$h(\lambda)=\frac{(N-2)^2}{4}\quad \forall\, 0\,\leq \lambda \leq \bar \lambda_N\,,$$
where $\frac{N-1}{4}\leq\bar \lambda_N \leq \lambda_1(\hn)$. Our results yield the further information: $\bar \lambda_3= \lambda_1(\mathbb{H}^{3})$ and $\bar \lambda_N< \lambda_1(\hn)$ for all $N>3$.
\end{rem}




Let $B(0,1)$ be the Euclidean unit ball and $\sigma: B(0,1) \rightarrow \hn$, where $\hn$ is the ball model for the hyperbolic space, be the conformal map. By defining

 \begin{equation}\label{transf}
 v(x) = \left( \frac{2}{1 - |x|^2} \right)^{\frac{N-2}{2}} u(\sigma(x))\quad x\in B(0,1)
\end{equation}
from Theorem \ref{poincare} we derive

\begin{cor}\label{cor}
Let $ N \geq 3.$ For all $ u \in C^{\infty}_{0}(B(0,1))$ the following inequality with optimal constants holds
\begin{equation*}
\int_{B(0,1)} |\nabla v|^2 \ dx - \frac{1}{4} \int_{B(0,1)} \left( \frac{2}{ 1 - |x|^2} \right)^2 v^2 \ dx \geq
\frac{1}{4} \int_{B(0,1)} \left( \frac{2}{ 1- |x|^2} \right)^2  \frac{v^2}{\left(\log\left( \frac{1 + |x|}{ 1- |x|}   \right) \right)^2}
\ dx,
\end{equation*}
where $dx$ denotes the Euclidean volume.
\end{cor}
 As far as we are aware this inequality is not known in literature and is a slight improvement upon an inequality proved in \cite[Theorem A]{GFT}, which is already sharp in a suitable sense, see Section \ref{hyperbolic proof}.

 Finally, in the spirit of \cite[Appendix B]{mancini}, we consider the upper half space model for $\hn$, namely $\mathbb{R}^{N}_{+} = \{ (x, y) \in \mathbb{R}^{N-1} \times \mathbb{R}^{+} \} $ endowed with the Riemannian metric $\frac{\delta_{ij}}{y^2}$. By exploiting the transformation
 \begin{equation} \label{halftransformation}
 v(x,y):= y^{-\frac{N-2}{2}} u(x,y), \ \ x \in \mathbb{R}^{N-1}, y \in \mathbb{R}^{+}\,.
 \end{equation}

for  $u \in C_c^{\infty}(\mathbb{H}^{N}),$ \eqref{poincareeq} yields an improved Hardy-Maz'ya inequality in the half space. Before stating it, we recall that the constant 1/4 in the Hardy-Maz'ya inequality
\begin{equation*}
\int_{\mathbb R^{+}} \int_{\mathbb{R}^{N-1}} |\nabla v|^2 \ dx \ dy \geq \frac{1}{4} \int_{\mathbb{R}^{+}} \int_{\mathbb{R}^{N-1}} \frac{v^2}{y^2} \ dx \ dy\,,
\end{equation*}
where $(x, y) \in \mathbb{R}^{N-1} \times \mathbb{R}^{+}$, is sharp, see \cite[2.1.6, Cor. 3]{Ma} and also \cite{FMT,FTT}.

\begin{cor}\label{corHALF1}
Let $N \geq 3.$ For all $v \in C_c^{\infty}(\mathbb{R}^{N}_{+})$ the following inequality holds
\begin{equation}\label{mazya1}
\int_{\mathbb R^{+}} \int_{\mathbb{R}^{N-1}} |\nabla v|^2 \ dx \ dy - \frac{1}{4} \int_{\mathbb{R}^{+}} \int_{\mathbb{R}^{N-1}} \frac{v^2}{y^2} \ dx \ dy \geq \frac{1}{4} \int_{\mathbb{R}^{+}} \int_{\mathbb{R}^{N-1}} \frac{v^2}{y^2 d^2} \ dx \ dy,
\end{equation}
where $(x, y) \in \mathbb{R}^{N-1} \times \mathbb{R}^{+}$ and $d: = \cosh^{-1} \left( 1 + \frac{(y - 1)^2 + |x|^2}{2 y} \right)$. The constant 1/4 in the r.h.s. of \eqref{mazya1} is sharp.
\end{cor}

\begin{rem}
It is easy to see that $d: = d((x,y),(0,1) ) \sim \log(1/y)$ as $y \rightarrow 0$.

\end{rem}

Using similar arguments we have the following improved Hardy-Maz'ya inequality, see \cite[Appendix B]{mancini} for further details.

\begin{cor}\label{mazya}
Let $N \geq 3$ and $k\in \mathbb{N}_+$. For all $v=v(x,y) \in C_c^{\infty}(\mathbb{R}^{N-1} \times \mathbb{R}^{k}),$ with $v(x,0)=0$ if $k=1$, the following inequality with optimal constants (in the sense of Corollary \ref{corHALF1}) holds

\begin{equation*}
\int_{\mathbb{R}^{k}} \int_{\mathbb{R}^{N-1}} |\nabla v|^2 \ dx \ dy \geq \frac{(k-2)^2}{4} \int_{\mathbb{R}^{k}} \int_{\mathbb{R}^{N-1}}\frac{v^2}{y^2} \ dx \ dy
+ \frac{1}{4}\int_{\mathbb{R}^{k}} \int_{\mathbb{R}^{N-1}} \frac{v^2}{y^2 d^2} \ dx \ dy,
\end{equation*}
where $d: = \cosh^{-1} \left( 1 + \frac{(|y| - 1)^2 + |x|^2}{2 |y|} \right). $
\end{cor}

The results of Theorem \ref{poincare} can be generalized to more general manifolds under suitable curvature assumptions, which allow for curvature being unbounded below and yield a stronger Hardy inequality in such cases. In fact we have the following

\begin{thm}\label{general manifold}
 Let $N\geq 3$ and $M$ be a Riemannian Manifold with a pole $o$ satisfying the assumptions

 \begin{equation}\label{condition1}
   Cut \{o \} = \phi.
 \end{equation}
   \begin{equation}\label{condition2}
 K_{R}(x) \leq - \frac{\psi^{\prime \prime}}{\psi}\quad \forall\, x\in M,
   \end{equation}
 where $K_R$ denotes sectional curvature in the radial direction, $\psi$ is a positive, $C^2$ function which is increasing and such that $\psi(0)=\psi''(0)=0$, $\psi'(0)=1$. Moreover we also require that
 \begin{equation}\label{condition3}
  (N-2)\psi^{\prime} +(N-1) r \psi^{\prime \prime} \geq 0.
 \end{equation}
Then for all $u \in C_c^{\infty}(M),$ there holds

\begin{equation}\label{generalHardy}\begin{aligned}
 \int_{M} |\nabla_{M} u|^2 -  &\frac{(N-1)}{4}  \int_{M} \left[ 2\frac{\psi^{\prime \prime}}{\psi}+ (N-3)\frac{(\psi^{\prime 2}-1)}{\psi^2}\right] u^2\\ & \geq \frac{1}{4} \int_{M} \frac{u^2}{r^2}+\frac{(N-1)(N-3)}{4}\int_{M} \frac{u^2}{\psi^2}.\end{aligned}
\end{equation}
In particular \eqref{generalHardy} holds when $M$ is a Cartan-Hadamard manifold and condition \eqref{condition2} holds with $\psi$ a convex function satisfying  $\psi(0)=\psi''(0)=0$, $\psi'(0)=1$.

Assumption \eqref{condition3} is not required if $M$ coincides with the Riemannian model with pole $o$ defined by $\psi$ (see Section \ref{proof1}).

Finally, \eqref{generalHardy} is sharp in the following sense: given any  non negative function $\psi$ s.t. $\psi(0)=\psi''(0)=0$, $\psi'(0)=1$, the operator
\begin{equation}\label{crit}
-\Delta_{M} -  \frac{(N-1)}{4}  \left[ 2\frac{\psi^{\prime \prime}}{\psi}+ (N-3)\frac{(\psi^{\prime 2}-1)}{\psi^2}\right]- \frac{1}{4r^2}-\frac{(N-1)(N-3)}{4\psi^2}
\end{equation}
is critical on the Riemannian model corresponding to $\psi$, on which of course the curvature condition \eqref{condition2} holds as an equality.
\end{thm}

Of course we recover our first result when we consider the model manifold corresponding to $\psi=\sinh r$, which is well-known to coincide with the hyperbolic space.

We also comment that the quantities appearing in the second integral in the l.h.s. of \eqref{generalHardy} have a geometrical meaning: in fact,
\[
  K_{\pi,r}^{rad} = - \frac{\psi^{\prime \prime}}{\psi}\quad \text{and}  \quad H_{\pi,r}^{tan} = - \frac{(\psi^{\prime})^2 - 1}{\psi^{2}}
 \]
where $K_{\pi,r}^{rad}$ (resp. $H_{\pi,r}^{tan}$) denote sectional curvature relative to planes containing (resp. orthogonal to) the radial direction in the Riemannian model associated to $\psi$, see Section \ref{proof1} for some further detail.

\begin{ex} The weight in the second term in the l.h.s. of \eqref{generalHardy} can be unbounded. Consider e.g. a Riemannian model associated to a function $\psi$ satisfying $\psi(r)\sim e^{r^a}$ as $r\to+\infty$, where $a>1$. Then sectional curvatures are unbounded below and one has:
\[
-  \frac{(N-1)}{4}  \left[ 2\frac{\psi^{\prime \prime}}{\psi}+ (N-3)\frac{(\psi^{\prime 2}-1)}{\psi^2}\right]\sim -\frac{(N-1)^2a^2}4 r^{2a-2}\ \textrm{as}\ r\to+\infty.
\]
Hence
\[
 \int_{M} |\nabla_{M} u|^2 -  \int_{M} w\, u^2 \geq \frac{1}{4} \int_{M} \frac{u^2}{r^2}+\frac{(N-1)(N-3)}{4}\int_{M} \frac{u^2}{\psi^2}.
\]
where \[
w(r)\sim \frac{(N-1)^2a^2}4 r^{2a-2}\ \textrm{as}\ r\to+\infty.
\]
\end{ex}

Some complementary results can be given on bounded domains in $\hn$. When restricted to a bounded domain, the operator $P$ defined in Remark \ref{remarkoptimal} is clearly not critical anymore; in Proposition \ref{infinity} we show that our methods immediately provide an infinite expansion of logarithmic weight that can be added to \eqref{poincareeq} when posed on geodesic balls. Before giving a precise statement, we first introduce some auxiliary functions, which are basically the iterated log functions arising in several paper in the euclidean setting, see for instance \cite{FT}. Let $X_{1}(t) = (1 - \log t)^{-1}$ for $t \in (0, 1].$ We define recursively the functions:
\[
X_{k}(t) = X_{1}(X_{k-1}(t)), \ \quad k = 2, 3 \ldots
\]
The $X_{k}$ are well defined and that for $k = 1,2 \ldots$ one has
\[
X_{k}(0) = 0, \quad \quad X_{k} (1) = 1, \quad \quad 0 < X_{k}(t) < 1, \quad \ \mbox{for} \ t \in (0,1).
\]

We denote as before $r := \rho(x, o)$ and we prove

\begin{prop}\label{infinity}
Let $B:= B(o,1) \subset \hn$ be a geodesic ball of radius 1 and  $N \geq 3.$ Then for every $u \in C_{c}^{\infty}(B)$ there holds

\begin{align}\label{infinityi}
\int_{B} |\nabla_{\hn} u|^2 \ dv_{\hn} - \frac{(N-1)^2}{4} \int_{B} u^2 \ dv_{\hn} &
\geq \frac{1}{4} \int_{B} \frac{u^2}{r^2} \ dv_{\hn} + \frac{(N-1)(N-3)}{4} \int_{B} \frac{u^2}{\sinh^2 r} \ dv_{\hn} \notag \\
& + \frac{1}{4} \sum_{i = 1}^{\infty} \int_{B} \frac{u^2}{r^2} X_{1}^2(r)X_{2}^2(r) \ldots X_{i}^2(r)  \ dv_{\hn}.
\end{align}

Moreover, for each $k = 1,2 \ldots $ the latter constant is the best constant for the corresponding $k-$ improved inequality, that is
\[
\frac{1}{4} = \inf_{u \in C_{c}^{\infty}(B)} \frac{\langle Pu,u \rangle
- \frac{1}{4} \sum_{i = 1}^{k-1} \int_{B} \frac{1}{r^2} X_{1}^2(r)X_{2}^2(r) \ldots X_{i}^2(r) u^2 \ dv_{\hn} }{\int_{B} \frac{1}{r^2} X_{1}^2 X_{2}^2 \ldots X_{k}^2 u^2 \ dv_{\hn}},
\]
where $\langle Pu, u \rangle := \int_{B} |\nabla_{\hn} u|^2 \ dv_{\hn} - \frac{(N-1)^2}{4} \int_{B} u^2 \ dv_{\hn} -
\frac{1}{4} \int_{B} \frac{u^2}{r^2} \ dv_{\hn} - \frac{(N-1)(N-3)}{4} \int_{B} \frac{u^2}{\sinh^2 r} \ dv_{\hn}.$
\end{prop}
\par\medskip\par

 \section{Poincar\'{e}-Rellich Inequalities}\label{Rellich}

In this section we state our Poincar\'{e}-Rellich inequality on the hyperbolic space and related Euclidean inequalities. First we have
\begin{thm}\label{TPR}
Let $ N \geq 5$. For all $ u \in C^{\infty}_{c}(\mathbb{H}^{N})$ there holds

\begin{equation}
\begin{aligned}\label{PR}
 \int_{\hn} |\Delta_{\hn} u|^2 \ dv_{\hn} - \frac{(N-1)^4}{16} \int_{\hn} u^2 \ dv_{\hn}
&\geq \frac{(N-1)^2}{8} \int_{\hn} \frac{u^2}{r^2} \ dv_{\hn}+\frac{9}{16} \int_{\hn} \frac{u^2}{r^4} \ dv_{\hn}\\ &
+\frac{(N^2-1)(N-3)^2}{8} \int_{\hn} \frac{u^2}{\sinh^2 r} \  dv_{\hn}\\ &+\frac{(N-1)(N-3)(N^2-4N-3)}{16} \int_{\hn} \frac{u^2}{\sinh^4 r} \ dv_{\hn}
\end{aligned}
\end{equation}
The constant $\frac{(N-1)^4}{16}$  is of course sharp in the sense that the l.h.s. of \eqref{PR} can be negative if such constant is replaced by a larger one.
Furthermore, the constant $\frac{(N-1)^2}{8}$ appearing in \eqref{PR} is sharp in the sense that no inequality of the form
\[
\int_{\hn} |\Delta_{\hn} u|^2 \ dv_{\hn} - \frac{(N-1)^4}{16} \int_{\hn} u^2 \ dv_{\hn}
\geq c\, \int_{\hn} \frac{u^2}{r^2} \ dv_{\hn}
\]
holds for all $ u \in C^{\infty}_{c}(\mathbb{H}^{N})$ when $c>\frac{(N-1)^2}{8}$.
\end{thm}


\begin{rem}[Joint sharpness of some of the constants]\label{rellichsharpness}
The multiplicative constants appearing in two of the terms in the r.h.s. of \eqref{PR}, namely:
\[
\frac{9}{16} \int_{\hn} \frac{u^2}{r^4} \ dv_{\hn}+\frac{(N-1)(N-3)(N^2-4N-3)}{16} \int_{\hn} \frac{u^2}{\sinh^4 r} \ dv_{\hn}
\]
are \it jointly sharp\rm. By this we mean that the inequality
\begin{equation*}
\begin{aligned}
 \int_{\hn} &|\Delta_{\hn} u|^2 \ dv_{\hn} - \frac{(N-1)^4}{16} \int_{\hn} u^2 \ dv_{\hn}\ge c_1 \int_{\hn} \frac{u^2}{r^4} \ dv_{\hn}
+c_2 \int_{\hn} \frac{u^2}{\sinh^4 r} \ dv_{\hn}
\end{aligned}
\end{equation*}
cannot hold for all $ u \in C^{\infty}_{c}(\mathbb{H}^{N})$, or even for for all $ u \in C^{\infty}_{c}(B_\varepsilon)$ given any $\varepsilon>0$, if
\[\begin{aligned}
&c_1=\frac9{16},\ \ \ c_2>\frac{(N-1)(N-3)(N^2-4N-3)}{16}\ \ \ \textrm{or}\\
&c_1>\frac9{16},\ \ \ c_2=\frac{(N-1)(N-3)(N^2-4N-3)}{16}.
\end{aligned}
\]
This is a consequence of the following elementary facts:
\[
\frac9{16}+\frac{(N-1)(N-3)(N^2-4N-3)}{16}=\frac{N^2(N-4)^2}{16}
\]
and the r.h.s. is the known best constant for the standard $N$ dimensional Euclidean Rellich inequality, both on the whole ${\mathbb R}^N$ or in any open set containing the origin. The claim follows by noticing that $\sinh r\sim r$ as $r\to0$.

We refer to Remark \ref{conj} for a discussion of the possible sharpness of the constant 9/16 found here. Clearly, should this value be sharp, sharpness of the constant $(N-1)(N-3)(N^2-4N-3)/16$ in an obvious sense would then follow as well by the above discussion.
\end{rem}

\par
  We give a sample of the several Euclidean inequalities which can possibly be deduced from Theorem \ref{TPR}. We consider e.g. the half space model for $\hn$ exploiting the transformations 

\begin{equation}\label{higherorder}
  v(x,y):= y^{-\alpha} u(x,y), \ \ x \in \mathbb{R}^{N-1}, y \in \mathbb{R}^{+},
\end{equation}
with $\alpha=(N-4)/2$ or $\alpha=(N-2)/2$
from \eqref{PR} we derive the following statement.

\begin{cor}\label{cor2}
Let $N \geq 5.$ For all  $v \in C_c^{\infty}(\mathbb{R}^{N}_{+})$ the following inequalities hold

\begin{align}\label{HALFrellich1}
& \int_{\mathbb{R}^{+}} \int_{\mathbb{R}^{N-1}} \left( y^2 (\Delta v)^2 + \frac{N(N-2)}{2} |\nabla v|^2 \right) \ dx \ dy
 \geq \frac{2N^2 - 4N +1}{16} \int_{\mathbb{R}^{+}} \int_{\mathbb{R}^{N-1}} \frac{v^2}{y^2} \ dx \ dy \notag \\
& + \frac{(N-1)^2}{8} \int_{\mathbb{R}^{+}} \int_{\mathbb{R}^{N-1}}
\frac{v^2}{y^2 d^2} \ dx \ dy + \frac{9}{16} \int_{\mathbb{R}^{+}} \int_{\mathbb{R}^{N-1}} \frac{v^2}{y^2 d^4} \ dx \ dy
\end{align}
\emph{and}

\begin{align}\label{HALFrellich2}
& \int_{\mathbb{R}^{+}} \int_{\mathbb{R}^{N-1}} \left( (\Delta v)^2 + \frac{(N^2 - 2N - 4)}{2} \frac{|\nabla v|^2}{y^2} \right) \ dx \ dy
 \geq \frac{9}{16} (2N^2 - 4N - 7) \int_{\mathbb{R}^{+}} \int_{\mathbb{R}^{N-1}} \frac{v^2}{y^4} \ dx \ dy \notag \\
& + \frac{(N-1)^2}{8} \int_{\mathbb{R}^{+}} \int_{\mathbb{R}^{N-1}}
\frac{v^2}{y^4 d^2} \ dx \ dy + \frac{9}{16} \int_{\mathbb{R}^{+}} \int_{\mathbb{R}^{N-1}} \frac{v^2}{y^4 d^4} \ dx \ dy.
\end{align}
Furthermore, the constants in \eqref{HALFrellich1} satisfy the following optimality properties: \par \smallskip\par
$\bullet$ no inequality of the form
$$\int_{\mathbb{R}^{+}} \int_{\mathbb{R}^{N-1}} \left( y^2 (\Delta v)^2 +
c |\nabla v|^2 \right) \ dx \ dy
 \geq \frac{2N^2 - 4N +1}{16} \int_{\mathbb{R}^{+}} \int_{\mathbb{R}^{N-1}} \frac{v^2}{y^2} \ dx \ dy $$
holds for all $ u \in C^{\infty}_{c}(\mathbb{H}^{N})$ when $c< \frac{N(N-2)}{2}$;
 \par \smallskip\par
$\bullet$ no inequality of the form
$$\int_{\mathbb{R}^{+}} \int_{\mathbb{R}^{N-1}} \left( y^2 (\Delta v)^2 + \frac{N(N-2)}{2} |\nabla v|^2 \right) \ dx \ dy
 \geq c \int_{\mathbb{R}^{+}} \int_{\mathbb{R}^{N-1}} \frac{v^2}{y^2} \ dx \ dy $$
holds for all $ u \in C^{\infty}_{c}(\mathbb{H}^{N})$ when $c>  \frac{2N^2 - 4N +1}{16} $;
 \par \smallskip\par
$\bullet$ no inequality of the form
\begin{align*}
& \int_{\mathbb{R}^{+}} \int_{\mathbb{R}^{N-1}} \left( y^2 (\Delta v)^2 + \frac{N(N-2)}{2} |\nabla v|^2 \right) \ dx \ dy
 \geq \frac{2N^2 - 4N +1}{16} \int_{\mathbb{R}^{+}} \int_{\mathbb{R}^{N-1}} \frac{v^2}{y^2} \ dx \ dy \notag \\
& + c \int_{\mathbb{R}^{+}} \int_{\mathbb{R}^{N-1}}
\frac{v^2}{y^2 d^2} \ dx \ dy
\end{align*}
holds for all $ u \in C^{\infty}_{c}(\mathbb{H}^{N})$ when $c> \frac{(N-1)^2}{8} $.
\par
Similarly, the constants $ \frac{(N^2 - 2N - 4)}{2},$ $\frac{9}{16}(2N^2 - 4N -7)$ and  $\frac{(N-1)^2}{8}$
in \eqref{HALFrellich2} are optimal in the above sense.

 \end{cor}

  \section{Proof of the Poincar\'e-Hardy inequality  \eqref{poincareeq} and of Theorem \ref{general manifold}}\label{proof1}

  We shall first state, also for later use, a result on \it Riemannian models\rm, namely an $N$-dimensional Riemannian manifold admitting a pole  $o$ and whose metric is given in spherical coordinates by
 \begin{equation}\label{meetric}
 ds^2 = dr^2 + \psi^2(r) d\omega^2,
 \end{equation}
 where $d\omega^2$ is the metric on sphere $\mathbb{S}^{N-1}$ and $\psi$ is a $C^{\infty}$ nonnegative function on $[0,\infty),$
 strictly positive on $(0, \infty)$ such that $\psi(0) = \psi^{\prime \prime}(0)= 0$ and $\psi^{\prime}(0) = 1$. The coordinate $r$ represents the Riemannian distance from the pole $o,$ see e.g. \cite{RGR, PP} for further details. It is well known that there exist an orthonormal
 frame $\{F_{j} \}_{j = 1, \ldots, N}$ on $M,$ where
 $F_{N}$ corresponds to the radial coordinate, and $F_{1}, \ldots, F_{N-1}$ to the spherical cordinates, for which
 $F_{i} \wedge F_{j}$ diagonalize the curvature operator $\mathcal{R}$ :

 \[
  \mathcal{R} (F_{i} \wedge F_{N}) = - \frac{\psi^{\prime \prime}}{\psi} F_{i} \wedge F_{N}, \ i < N,
 \]
 \[
  \mathcal{R} (F_{i} \wedge F_{j}) = - \frac{(\psi^{\prime})^2 - 1}{\psi^2} F_{i} \wedge F_{j}, \ i,j < N.
 \]

 The following quantities
 \begin{equation}\label{curvature}
  K_{\pi,r}^{rad} = - \frac{\psi^{\prime \prime}}{\psi}\quad \text{and}  \quad H_{\pi,r}^{tan} = - \frac{(\psi^{\prime})^2 - 1}{\psi^{2}}
 \end{equation}
then coincide with the sectional curvature w.r.t. planes containing the radial direction and, respectively, orthogonal to it.  

Notice that the Riemannian Laplacian of a scalar function $\Phi$ on $M$ is given by
  \begin{equation}\label{laplacian}
  \Delta_{g} \Phi (r, \theta_{1}, \ldots, \theta_{N-1})  =
 \frac{1}{\psi^2} \frac{\partial}{\partial r} \left[ (\psi(r))^{N-1} \frac{\partial \Phi}{\partial r}(r, \theta_{1},
 \ldots, \theta_{N-1}) \right] \\
  + \frac{1}{\psi^2} \Delta_{\mathbb{S}^{N-1}} \Phi(r, \theta_{1}, \ldots, \theta_{N-1}),
 \end{equation}
 where $\Delta_{\mathbb{S}^{N-1}}$ is the Riemannian Laplacian on the unit sphere $\mathbb{S}^{N-1}.$ In particular,
 for radial functions, namely functions depending only on $r$, one has

 \begin{equation}\label{radlaplacian}
 \Delta_{g} \Phi(r) = \frac{1}{(\psi(r))^{N-1}} \frac{\partial}{\partial r} \left[ (\psi(r))^{N-1} \frac{\partial \Phi}{\partial r}
 (r) \right] =  \Phi^{\prime \prime}(r) + (N-1) \frac{\psi^{\prime}(r)}{\psi(r)} \Phi^{\prime}(r),
 \end{equation}
 where from now on a prime will denote, for radial functions, derivative w.r.t $r$. Note that the quantity
 $(N-1)\frac{\psi^{\prime}(r)}{\psi(r)}$ has a geometrical meaning, namely it represents the mean curvature of the geodesic sphere
 of radius $r$ in the radial direction. 
%
%
%

We are now ready to state the following result:
\begin{prop}\label{maintheorem}

Let $ N \geq 3$ and $M$ as given in \eqref{meetric}. For all $ u \in C^{\infty}_{c}(M)$ there holds

\begin{equation*}
\begin{split}
\int_{M} |\nabla_{g} u|^2 \ dv_{g} +
\frac{(N-1)}{4} \int_{M} \left[ 2 K_{\pi,r}^{rad} + (N-3) H_{\pi,r}^{tan} \right] u^2 \ dv_{g}
\\
\geq \frac{1}{4} \int_{M} \frac{u^2}{r^2} \ dv_{g} + \frac{(N-1)(N-3)}{4} \int_{M} \frac{u^2}{\psi^2} \ dv_{g} \,.
\end{split}
 \end{equation*}

\end{prop}

First we prove some preliminary results which are useful to define a supersolution to a suitable pde. We took inspiration from \cite{BD} where a similar construction was applied in a completely different setting.

  \begin{lem}\label{firstlem}
  Let $\Phi(r) = \left ( \frac{\psi(r)}{r} \right)^{\alpha},$ where $\alpha$ is real parameter, then $\Phi$ satisfies the following
  equation:

  \begin{align*}
  - \Delta_{g} \Phi - \alpha\left[ K_{\pi,r}^{rad} + (\alpha -2+N) H_{\pi,r}^{tan} \right] \Phi
  & = - \alpha (\alpha -2+N)  \frac{\Phi}{\psi^2} \\ \notag
  &  - \frac{\alpha(\alpha + 1)}{r^2} \Phi + \frac{2 \alpha^2 + \alpha (N-1)}{r} \frac{\psi^{\prime}}{\psi} \Phi.
  \end{align*}
 Hence, $\Phi(r) = (\frac{r}{\psi(r)})^{\frac{N-1}{2}}$ satisfies
  \begin{align*}
  - \Delta_{g} \Phi + \frac{(N-1)}{4} \left[ 2 K_{\pi,r}^{rad} + (N-3) H_{\pi,r}^{tan} \right]\Phi & =
  \frac{(N-1)(N-3)}{4} \Phi\left(\frac{1}{\psi^2} - \frac{1}{r^2}\right).
  \end{align*}
  \end{lem}

  \begin{proof}
    The expression of the Riemannian Laplacian \eqref{radlaplacian}, enables us to write
  \begin{equation}\label{radl1}
   - \Delta_{g} \Phi = - \Phi^{\prime \prime} - (N-1) \frac{\psi^{\prime}}{\psi} \Phi'.
  \end{equation}

 It is easy to see that

  \begin{equation}\label{diff1}
  \Phi^{\prime}(r) = \alpha \left(\frac{\psi}{r} \right)^{\alpha - 1}
  \left[ \frac{\psi^{\prime}}{r} - \frac{\psi}{r^2} \right],
  \end{equation}
   and
  \begin{align}\label{diff2}
  \Phi^{\prime \prime}(r) & = \alpha(\alpha - 1) \left( \frac{\psi}{r} \right)^{\alpha -2}
  \left[ \frac{\psi^{\prime^2}}{r^2}  + \frac{\psi^2}{r^4} - 2 \frac{\psi \psi^{\prime}}{r^3} \right]
   + \alpha \left( \frac{\psi}{r} \right)^{\alpha - 1} \left[ \frac{\psi^{\prime \prime}}{r} - 2 \frac{\psi^{\prime}}{r^2}
   + 2 \frac{\psi}{r^3} \right].
  \end{align}

  Now we can compute \eqref{radl1}, using \eqref{diff1} and \eqref{diff2},
  \begin{align*}
   - \Phi^{\prime \prime}(r) - (N-1) \frac{\psi^{\prime}}{\psi} \Phi' & =
   - \alpha(\alpha - 1) \left( \frac{\psi}{r} \right)^{\alpha -2}
  \left[ \frac{\psi^{\prime^2}}{r^2}  + \frac{\psi^2}{r^4} - 2 \frac{\psi \psi^{\prime}}{r^3} \right] \\
  & - \alpha \left( \frac{\psi}{r} \right)^{\alpha - 1} \left[ \frac{\psi^{\prime \prime}}{r} - 2 \frac{\psi^{\prime}}{r^2}
   + 2 \frac{\psi}{r^3} \right] \\
  & - \alpha (N-1) \left( \frac{\psi}{r} \right)^{\alpha - 1} \left[ \frac{\psi^{\prime^2}}{r \psi} - \frac{\psi^{\prime}}{r^2} \right]\\
  & = \left[ - (\alpha(\alpha - 1) + \alpha(N-1) ) \frac{((\psi^{\prime})^2 - 1)}{\psi^2} -
  \alpha \frac{\psi^{\prime \prime}}{\psi} \right] \Phi\\
  & - (\alpha(\alpha - 1) + \alpha (N-1)) \frac{\Phi}{\psi^2} - \frac{\alpha(\alpha + 1)}{r^2} \Phi \\
  & + \frac{(2 \alpha^2 + \alpha(N-1))}{r} \frac{\psi^{\prime}}{\psi} \Phi.
  \end{align*}
 The proof is concluded using formulas \eqref{curvature}.
\end{proof}

 \begin{prop}\label{firstprop}
  Let $f : M \backslash \{ {o} \} \rightarrow \mathbb{R}$ be a smooth radial function and
  $\Phi(r) = (\frac{r}{\psi(r)})^{\frac{N-1}{2}},$ then $\tilde\Phi(r) = \Phi(r) f(r)$ satisfies
  \begin{equation*}\begin{aligned}
   - \Delta_{g} \tilde\Phi + \frac{(N-1)}{4} \left[ 2 K_{\pi,r}^{rad} + (N-3) H_{\pi, r}^{tan} \right] \tilde\Phi  & =
  \frac{(N-1)(N-3)}{4} \frac{\tilde\Phi}{\psi^2} - \frac{(N-1) (N-3)}{4} \frac{\tilde\Phi}{r^2} \\
  & - (f''+\frac{N-1}r f')\Phi.
  \end{aligned}
  \end{equation*}
  \end{prop}

  \begin{proof}
  From the expression of Riemannian Laplacian on $M$ for radial function, we easily conclude

  \begin{align*}
   - \Delta_{g} \tilde\Phi(r) & = (-\Delta_{g} \Phi(r))f(r) - 2 \Phi^{\prime}(r)f^{\prime}(r) - \Phi(r)f^{\prime \prime}(r) \\
   & - (N-1)\frac{\psi^{\prime}(r)}{\psi(r)} \Phi(r) f^{\prime}(r).
  \end{align*}
  Now, using Lemma \ref{firstlem}, we have

\begin{align*}
- \Delta_{g} \tilde\Phi(r) +  \frac{(N-1)}{4} \left[ 2 K_{\pi,r}^{rad} + (N-3) H_{\pi, r}^{tan} \right]  \tilde\Phi(r) & =
\frac{(N-1)(N-3)}{4} \frac{\tilde\Phi}{\psi^2} \\ \notag
& - \frac{(N-1)(N-3)}{4} \frac{\tilde\Phi}{r^2} + (N-1) \frac{\psi^{\prime}}{\psi} \Phi(r)f^{\prime}(r) \\ \notag
& - \frac{(N-1)}{r} \Phi(r)f^{\prime}(r) - \Phi(r)f^{\prime \prime}(r)\\ \notag
& - (N-1) \frac{\psi^{\prime}}{\psi} \Phi(r) f^{\prime}(r),
\end{align*}
and hence we have the result.
\end{proof}

\noindent \textit{Proof of Proposition \ref{maintheorem} completed.}

The proof is based on supersolution technique. If we choose $f(r) = r^{\frac{(2-N)}{2}}$ in Proposition \ref{firstprop}, then $\tilde u(r):=\left ( \frac{r}{\psi(r)} \right)^{\frac{N-1}{2}} r^{\frac{(2-N)}{2}}$ satisfies
\begin{align*}
  - \Delta_{g} \tilde u + \frac{(N-1)}{4} \left[ 2 K_{\pi,r}^{rad} + (N-3) H_{\pi, r}^{tan} \right] \tilde u  & =
  \frac{(N-1)(N-3)}{4} \frac{\tilde u}{\psi^2}
   + \frac{1}{4} \frac{\tilde u}{r^2}
\end{align*}
Hence, $\tilde u(r)$ is a \emph{supersolution} of
\begin{equation*}
- \Delta_{g} u + \frac{(N-1)}{4} \left[ 2 K_{\pi,r}^{rad} + (N-3) H_{\pi,r}^{tan} \right] u - \frac{u}{4r^2} = 0.
\end{equation*}
Then, since $\tilde u(r) \in H^{1}_{\emph{loc}} (M \setminus \{ o\})$ and $\frac{1}{r^2},\frac{1}{\psi^2}\in L^{1}_{\emph{loc}} (M )$, \cite[Theorem 1.5.12]{DA} applies and the result follows, in principle for functions supported away from $o$, then by approximation.\qed
\par \medskip\par

\noindent\it Proof of the Poincar\'e-Hardy inequality \eqref{poincareeq} and of some further statements of Theorem \ref{poincare}. \rm \par \smallskip\par Inequality \eqref{poincareeq} follows from Proposition \ref{maintheorem} noticing that the hyperbolic space coincides with the model manifold associated to $\psi(r)=\sinh r$ and  in that case $K_{\pi, r}^{rad} = -1$ and $H_{\pi,r}^{tan}  = -1$. Hence, the operator $H$ defined in the statement of Theorem \ref{poincare} is nonnegative. To prove that $H$ is critical we show that the equation $Hu=0$ admits a ground state in $\hn\setminus \{o\}$, namely a positive solution of minimal growth in a neighborhood of infinity in $\hn\setminus \{o\}$, see \cite[Section 1]{PT2}. From Proposition \ref{firstprop} two linearly independent solutions of the equation $Hu=0$ are given explicitly by $v_{\pm}(r)=u_{\pm}(r)  \Phi(r)$, where $\Phi(r) = \left( \frac{r}{\sinh r} \right)^{\frac{N-1}{2}}$ and $u_{\pm}$ are two linearly independent solutions of the Euler equation
$$
-  u^{\prime \prime} - \frac{N-1}{r} u^{\prime} = C_H\,  \frac{u}{r^2}\,,
$$
where $C_{H} = \frac{(N-2)^2}{4}$ is the well known Hardy constant. Then $u_{+}(r) = r^{\frac{(2 - N)}{2}}$ and $u_{-}(r) = r^{\frac{(2 - N)}{2}} \log(r^{2-N})$ hence
$v_{+}$ is a positive global solution while $v_{-}$ changes sign. Since $| v_{-}(r)|$ is a positive solution of $H u = 0$ near infinity of $\hn\setminus\{o\}$ and
\[
\lim_{r \rightarrow 0} \frac{v_{+}(r)}{v_{-}(r)}=\lim_{r \rightarrow +\infty} \frac{v_{+}(r)}{ |v_{-}(r)|} = 0\,,
\]
by \cite[Proposition 6.1]{pinch} we conclude that $v_{+}$ is  a positive solution of minimal growth in a neighborhood of infinity in $\hn\setminus\{o\}$ and hence a ground
state of the equation $Hu=0$. Namely, $H$ is critical.\par
At last, the fact that the constant $(N-1)(N-3)/4$ is sharp in the sense described in the statement of Theorem \ref{poincare}, follows by noticing that
\[
\frac14+\frac{(N-1)(N-3)}4=\frac{(N-2)^2}4,
\]
that $\sinh r\sim r$ as $r\to0$, and that the best Hardy constant on a domain including the origin is $(N-2)^2/4$ whatever the domain is. \qed

\subsection{Hardy type inequality for general manifolds}

In this section we prove Theorem \ref{general manifold}. Before proceeding further we first recall some known facts.

Let $(M,g)$ be a Riemannian manifold. Take a point (pole) $o \in M$ and denote $Cut(o)$ the cut locus of $o.$ We can
define the polar coordinates in $M \setminus Cut^*(o),$ where $Cut^*(o) = Cut(o) \cup \{ o\}.$ Indeed, to any point
$x \in M \setminus Cut^*(o)$ we can associate the polar  radius $r(x) := \mbox{dist}(x,o)$ and the polar angle
$\theta \in \mathbb{S}^{N-1},$ such that the minimal geodesics from $o$ to $x$ starts at $o$ to the direction $\theta.$

The Riemannian metric $g$ in $M \setminus Cut^*(o)$ in the polar coordinates takes the form
\[
 ds^2 = d r^2 + a_{i,j}(r,\theta) d\theta_{i} \theta_{j},
\]
where $(\theta_{1},\ldots, \theta_{N-1})$ are coordinates on $\mathbb{S}^{N-1}$ and $((a_{i,j}))_{i,j= 1, \ldots, N}$ is a positive
definite Matrix.

Let $a: = \det(a_{i,j}),$ $B(o,\rho) = \{ x := (r, \theta): r  < \rho \}. $ Then in $M \setminus Cut^*(o)$ we have

\[
 \Delta_{M} = \frac{1}{\sqrt{a}} \frac{\partial}{\partial r}\left( \sqrt{a} \frac{\partial}{\partial r} \right)
+ \Delta_{\partial B(o,r)} = \frac{\partial^2}{\partial r^2} + m(r, \theta) \frac{\partial}{\partial r}
+ \Delta_{\partial B(o,r)},
\]
where $\Delta_{\partial B(o,r)}$ is the Laplace-Beltrami operator on the geodesic sphere $\partial B(o,r)$ and $m(r, \theta)$
is a smooth function on $(0, \infty) \times \mathbb{S}^{N-1}$ which represents the mean curvature of $\partial B(o,r)$ in the
radial direction.

Our result follows by standard Hessian comparison. We give some details for completeness and for the reader's convenience.

\begin{lem}\label{simplelemma}
Let $\tilde \Phi(r) := \left( \frac{r}{\psi(r)} \right)^{\frac{N-1}{2}} r^{\frac{2-N}{2}} $ for $r>0$. Then, $\tilde \Phi$ is non increasing if condition \eqref{condition3} holds.  
\end{lem}

\begin{proof}
 It is easy to see that
\[
 \frac{\partial \tilde \Phi}{\partial r} = \frac{1}{2\sqrt{r}} \psi^{-\frac{(N-1)}{2}} \left[ 1 - (N-1) r \frac{\psi^{\prime}}{\psi} \right]
\]
\[
 = \frac{1}{2\sqrt{r}} \psi^{- \frac{(N+1)}{2}} \left[ \psi - (N-1) r \psi^{\prime} \right].
\]
Let us define
\[
 p(r) : = \psi - (N-1)r \psi^{\prime},
\]
then
\[
 p^{\prime}(r) = -(N-2)\psi^{\prime} - (N-1) r \psi^{\prime \prime} \leq 0,
\]
hence by our hypothesis, we obtain the assertion.
\end{proof}

We recall a well known fact.

 \begin{lem}\label{comp}\cite{RGR,AG}
Let $M$ be a manifold with pole $o$ satisfying the assumptions \eqref{condition1}, \eqref{condition2}. Then
 \[
 m(r, \theta) \geq (N-1) \frac{\psi^{\prime}(r)}{\psi(r)}\quad \text{for all } r>0 \text{ and } \theta\in\mathbb{S}^{N-1}.
 \]
 \end{lem}

\noindent \it Proof of Theorem \ref{general manifold}\rm. Using Lemma \ref{comp}, the monotonicity property stated in Lemma \ref{simplelemma} and Proposition \ref{firstprop} with the choice $f(r) = r^{\frac{(2-N)}{2}}$, we get that $\tilde \Phi$ satisfies
$$
 -\Delta_{M} \tilde \Phi \geq 
 \frac{(N-1)}{4}\left[ 2\frac{\psi^{\prime \prime}}{\psi}+ (N-3)\frac{(\psi^{\prime 2}-1)}{\psi^2}\right] \tilde \Phi + \frac{(N-1)(N-3)}{4} \frac{\tilde \Phi}{\psi^2} +  \frac{1}{4} \frac{ \tilde \Phi}{r^2}\,.
$$

From the above calculations the proof of Theorem \ref{general manifold} follows at once by the supersolution method. Finally if M coincides with the Riemannian model with pole $o$ defined by $\psi$, then the above inequality becomes an equality.  Hence, using the arguments similar to that of Theorem~\ref{poincare} and exploiting once more Proposition \ref{firstprop}, one can show that
$\tilde \Phi$ is the ground state and thus proving the criticality of the resulting operator. \qed

\section{Alternative proof of optimality in Theorem \ref{poincare} and Proof of Corollary \ref{cor}}\label{hyperbolic proof}

Inequality \eqref{poincareeq} follows from Theorem \ref{maintheorem} with $\psi(r)=\sinh r$. In this section we give an alternative proof of the optimality of the constants using a suitable transformation. As a byproduct this will yield the proof of Corollary \ref{cor}.\par Let $C_{\hn}$ be the best constant in \eqref{poincareeq}, i.e

\begin{equation}\label{bestconstant1}
 C_{\hn} = \inf_{C^{\infty}_c (\mathbb{H}^{N})}
\frac{\int_{\hn} |\nabla_{\hn} u|^2 \ dv_{\hn} -  \frac{(N-1)^2}{4}\int_{\hn}  u^2 \ dv_{\hn}}{ \int_{\hn} \frac{u^2}{r^2} \ dv_{\hn}}.
\end{equation}
 Then clearly, from \eqref{poincareeq} it follows that $C_{\hn} \geq \frac{1}{4}.$  We shall show that $C_{\hn} \leq \frac{1}{4}$.

Let $B(0,1)$ be the Euclidean unit ball, $\hn$ be the ball model for the hyperbolic space and $\sigma: B(0,1) \rightarrow \hn$ be the conformal map. We recall the definition \eqref{transf}, namely
\[
 v(x) = \left( \frac{2}{1 - |x|^2} \right)^{\frac{N-2}{2}} u(\sigma(x))\quad x\in B(0,1)
\]
Then, it is easy to check that
\begin{equation}\label{conformal1}
 \int_{\hn} |\nabla_{\hn} u|^2 \ dv _{\hn} = \int_{B(0,1)} |\nabla v|^2 dx +
\frac{N(N-2)}{4} \int_{B(0,1)} \left( \frac{2}{1 - |x|^2} \right)^{2} v^2 \ dx,
\end{equation}
\begin{equation}\label{conformal2}
\int_{\hn} u^2 \ dv_{\hn} = \int_{B(0,1)} \left(\frac{1-|x|^2}{2} \right)^{N-2} v^2 \left( \frac{2}{1 - |x|^2} \right)^{N} \ dx
= \int_{B(0,1)} \left( \frac{2}{1 - |x|^2} \right)^2 v^2 \ dx,
\end{equation}
and
\begin{align}\label{conformal3}
\int_{\hn} \frac{u^2}{r^2} \ dv_{\hn} & = \int_{B(0,1)} \left( \frac{1-|x|^2}{2} \right)^{N-2}
\frac{v^2}{\left(\log\left( \frac{1 + |x|}{ 1- |x|}   \right) \right)^2} \left( \frac{2}{ 1- |x|^2} \right)^{N} \ dx \\
& = \int_{B(0,1)} \left( \frac{2}{1- |x|^2} \right)^2 \frac{v^2}{\left(\log\left( \frac{1 + |x|}{ 1- |x|}   \right) \right)^2}
\ dx. \notag
\end{align}
Now, substituting \eqref{conformal1}, \eqref{conformal2} and \eqref{conformal3} in \eqref{bestconstant1}, we have the following
inequality in the Euclidean space:

\begin{equation*}
\int_{B(0,1)} |\nabla v|^2 \ dx - \frac{1}{4} \int_{B(0,1)} \left( \frac{2}{ 1 - |x|^2} \right)^2 v^2 \ dx \geq
C_{\hn} \int_{B(0,1)} \left( \frac{2}{ 1- |x|^2} \right)^2  \frac{v^2}{\left(\log\left( \frac{1 + |x|}{ 1- |x|}   \right) \right)^2}
\ dx.
\end{equation*}
In particular, this proves Corollary \ref{cor}. On the other hand, the fact that
\begin{equation*}
 \left(\log\left( \frac{1 + |x|}{ 1- |x|}   \right) \right)^2 \leq {\left( 1- \log \left(\frac{1 - |x|}{2} \right) \right)^2}\,,
\end{equation*}
together with elementary computations,  gives

\begin{equation}\label{bestconstant3}
\int_{B(0,1)} |\nabla v|^2 - \frac{1}{4} \int_{B(0,1)} \frac{v^2}{d^2} \ dx \geq C_{\hn}
\int_{B(0,1)} \frac{v^2}{d^2(1 - \log(\frac{d}{2}))^2} \ dx,
 \end{equation}
 where $d(x) := d = \mbox{dist}(\partial B(0,1), x) = ( 1 - |x|).$ Comparing \eqref{bestconstant3} with \cite[Theorem A]{GFT}, we finally get
\[
 C_{\hn} \leq \frac{1}{4}.
\]
Hence, $C_{\hn} = \frac{1}{4}$ and we conclude.\qed

\section{Proof of Theorem \ref{TPR}}\label{rellich-proof}
In Section \ref{ineq} we prove the stated Poincar\'{e}-Rellich inequality by using orthogonal decomposition in spherical harmonics and a suitable 1-dimensional Hardy-type
inequality. Then, in Section \ref{optimal} we prove the optimality of the first constant and state some hints suggesting optimality of the latter.

\subsection{Proof of inequality \eqref{PR}} \label{ineq}
 We first prove the following 1-dimensional Hardy-type inequality.

\begin{lem}\label{hardytype}

For all $ u \in C^{\infty}_c(0,\infty)$ there holds

\begin{equation*}
 \int_{0}^{\infty} \frac{u^{\prime 2}}{\sinh^2 r} \ dr \geq \frac{9}{4} \int_{0}^{\infty} \frac{u^2}{\sinh^4 r} \ dr
+ \int_{0}^{\infty} \frac{u^2}{\sinh^2 r} \ dr.
\end{equation*}

\begin{proof}

The proof mainly relies on integration by parts. Let us put $u := w \sinh r,$ where $w \in C^{\infty}_{c}(0, \infty)$ and compute

\begin{equation}
\begin{aligned}\label{oned}
 \int_{0}^{\infty} \frac{{u^{\prime}}^2}{\sinh^2 r}  &= \int_{0}^{\infty} \left[ {w^{\prime}}^2 +
w^2 \frac{\cosh^2 r}{\sinh^2 r} +  2 w w^{\prime} \frac{\cosh r}{\sinh r} \right] \ dr \\
& = \int_{0}^{\infty} \left[ {w^{\prime}}^2 + w^2 + \frac{w^2}{\sinh^2 r} +
2 w w^{\prime} \frac{\cosh r}{\sinh r} \right] \ dr.
\end{aligned}
\end{equation}

Moreover,

\begin{align*}
\int_{0}^{\infty} w w^{\prime} \frac{\cosh r}{\sinh r} \ dr = - \int_{0}^{\infty}
\left( \frac{\cosh r}{\sinh r} \right)^{\prime} w^2 \ dr - \int_{0}^{\infty} w w^{\prime} \frac{\cosh r}{\sinh r} \ dr,
\end{align*}
and hence

\begin{equation}\label{oned1}
 2 \int_{0}^{\infty} w w^{\prime} \frac{\cosh r}{\sinh r} \ dr = \int_{0}^{\infty} \frac{w^2}{\sinh^2 r}.
\end{equation}

Now putting \eqref{oned1} in \eqref{oned} and using 1-dimensional Hardy inequality, we have

\begin{align*}
\int_{0}^{\infty} \frac{{u^{\prime}}^2}{\sinh^2 r} & =  \int_{0}^{\infty} \left[ {w^{\prime}}^2 + w^2 +
2 \frac{w^2}{\sinh^2 r} \right] \ dr \\
& \geq \frac{1}{4} \int_{0}^{\infty} \frac{w^2}{r^2} \ dr +  \int_{0}^{\infty} w^2 \ dr +
2 \int_{0}^{\infty}  \frac{w^2}{\sinh^2 r} \ dr \\
& \geq \frac{1}{4} \int_{0}^{\infty}  \frac{w^2}{\sinh^2 r} \ dr + \int_{0}^{\infty} w^2 \ dr +
2 \int_{0}^{\infty} \frac{w^2}{\sinh^2 r}  \ dr \  \ (\mbox{since,}\sinh r > r) \\
& = \frac{9}{4} \int_{0}^{\infty} \left[ \frac{w^2}{\sinh^2 r} + w^2 \right] \ dr\\
& = \frac{9}{4} \int_{0}^{\infty} \frac{u^2}{\sinh^4 r} \ dr + \int_{0}^{\infty} \frac{u^2}{\sinh^2 r} \ dr,
\end{align*}
this proving the claim.

%

\end{proof}
\end{lem}

Let us recall some informations on the spherical harmonics and Laplace-Beltrami operator on the hyperbolic space.
By \eqref{laplacian} with $\psi(r)=\sinh r$, the Laplace-Beltrami operator in spherical coordinates is given by
\[
 \Delta_{\hn} = \frac{\partial^2}{\partial r^2} + (N-1)\, \coth r
\frac{\partial}{\partial r} + \frac{1}{\sinh ^2 r} \Delta_{\mathbb{S}^{N-1}},
\]
where $\Delta_{\mathbb{S}^{N-1}}$ is the Laplace-Beltrami operator on the unit sphere $\mathbb{S}^{N-1}.$
If we write $u(x) = u(r,\sigma) \in C_c^{\infty}(\hn),$ $r \in [0, \infty), \sigma \in \mathbb{S}^{N-1},$ then
by \cite[Ch.4, Lemma 2.18]{ES}  we have that
\[
 u(x): = u(r, \sigma) = \sum_{n= 0}^{\infty} d_{n}(r) P_{n}(\sigma)
\]
in $L^2({\hn}),$ where $\{ P_{n}\}$ is a complete orthonormal system of spherical harmonics and
\[
 d_{n}(r) = \int_{\mathbb{S}^{N-1}} u(r, \sigma) P_{n}(\sigma) \ d\sigma.
\]
We note that the spherical harmonic $P_{n}$ of order $n$ is the restriction to $\mathbb{S}^{N-1}$ of a homogeneous harmonic
polynomial of degree $n.$ Now we recall the following

\begin{lem}\label{arm} \cite[Lemma 2.1]{MSS}
 Let $P_{n}$ be a spherical harmonic of order $n$ on $\mathbb{S}^{N-1}.$ Then for every $n \in \mathbb{N}_{0}$

\[
 \Delta_{\mathbb{S}^{N-1}} P_{n} = -(n^2 + (N-2)n)P_{n}.
\]
The values $\lambda_{n} := n^2 + (N-2)n$ are the eigenvalues of the Laplace-Beltrami operator $- \Delta_{\mathbb{S}^{N-1}}$
on $\mathbb{S}^{N-1}$ and enjoy the property $\lambda_{n} \geq 0$ and $\lambda_{0} = 0.$
The corresponding eigenspace consists of all the spherical harmonics of order $n$ and has dimension $d_{n}$ where $d_{0} = 1,$
$d_{1} = N$ and
$$
 d_{n} = \left(\begin{array}{c}
N+n-1\\
n
\end{array} \right) - \left(\begin{array}{c}
N+n-3\\
n-2\
\end{array}\right),
$$
for $n\geq 2.$
\end{lem}

From Lemma \ref{arm} it is easy to see that

\begin{equation*}
 \Delta_{\hn} u(r, \sigma) = \sum_{n= 0}^{\infty} \left( d_{n}^{\prime \prime}(r) +
(N-1) \coth r d_{n}^{\prime} (r) - \frac{\lambda_{n} d_{n}(r)}{\sinh^2 r} \right) P_{n}(\sigma).
\end{equation*}

Let $u \in C^{\infty}_c(\hn)$ and  make the following transformation
\[
 v = (\sinh r)^{\frac{N-1}{2}} u.
\]
Then

\begin{equation*}
 \Delta_{\hn} v = \left( \frac{\partial^2}{\partial r^2} + (N-1) \coth r \frac{\partial}{\partial r} +
\frac{1}{\sinh^2r} \Delta_{\mathbb{S}^{N-1}}\right) (\sinh r)^{\frac{N-1}{2}} u.
\end{equation*}

We compute:

\begin{align*}
  \Delta_{\hn} v& = \frac{(N-1)(N-3)}{4} (\sinh r)^{\frac{N-1}{2}} \coth^2 r \ u + (N-1) (\sinh r)^{\frac{N-3}{2}} \cosh r
\frac{\partial u}{\partial r} \\
& + \frac{(N-1)}{2} (\sinh r)^{\frac{N-1}{2}} u + (\sinh r)^{\frac{N-1}{2}} \frac{\partial^2 u}{\partial r^2} \\
& + \frac{(N-1)^2}{2} \coth^2 r (\sinh r)^\frac{N-1}{2} u + (N-1) \coth r (\sinh r)^{\frac{N-1}{2}} \frac{\partial u}{\partial r} \\
& + (\sinh r)^{\frac{N-1}{2}} \frac{1}{\sinh^2 r} \Delta_{\mathbb{S}^{N-1}} u\\
& = (\sinh r)^{\frac{N-1}{2}} \left[ \frac{\partial^2 u}{\partial r^2} + (N-1) \coth r \frac{\partial u}{\partial r}
+ \frac{1}{\sinh^2 r} \Delta_{\mathbb{S}^{N-1}} u \right] \\
& + \left[ \frac{(N-1)(N-3)}{4} + \frac{(N-1)^2}{2} \right] \coth^2 r (\sinh r)^{\frac{N-1}{2}} u +
\frac{(N-1)}{2} (\sinh r)^{\frac{N-1}{2}} u \\
&+ (N-1) (\sinh r)^{\frac{N-3}{2}} \cosh r \left[ \frac{\partial}{\partial r} ((\sinh r)^{- \frac{(N-1)}{2}} v) \right]\\
& = (\sinh r)^{\frac{N-1}{2}} (\Delta_{\hn} u) +  \left[ \frac{(N-1)(N-3)}{4} + \frac{(N-1)^2}{2} \right] \coth^2 r v \\
& + \frac{(N-1)}{2} v - \frac{(N-1)^2}{2} \coth^2 r v + (N-1) \coth r \frac{\partial v}{\partial r} \\
& = (\sinh r)^{\frac{N-1}{2}} (\Delta_{\hn} u) + \frac{(N-1)(N-3)}{4} \coth^2 r v + \frac{(N-1)}{2} v + (N-1) \coth r
\frac{\partial v}{\partial r}\,.
\end{align*}
Hence, we have
\begin{equation}
 \begin{aligned}\label{trans1}
  \Delta_{\hn} u  &= \frac1{(\sinh r)^{\frac{(N-1)}{2}}} \left[ \Delta_{\hn} v -
 \left( \frac{(N-1)(N-3)}{4} \coth^2 r + \frac{(N-1)}{2} \right) v \right.  \left. -(N-1)\coth r \frac{\partial v}{\partial r} \right] \\
&= \frac1{(\sinh r)^{\frac{(N-1)}{2}}} \left[ \frac{\partial^2 v}{\partial r^2} - \left( \frac{(N-1)(N-3)}{4} \coth^2 r
+ \frac{(N-1)}{2} \right) v \right. \left. + \frac{1}{\sinh^2 r} \Delta_{\mathbb{S}^{N-1}} v \right].
 \end{aligned}
 \end{equation}

Now, expanding $v$ in the spherical harmonics
\[
 v(x) := v(r, \sigma) = \sum_{n=0}^{\infty} d_{n}(r) P_{n}(\sigma)
\]
and putting this in \eqref{trans1}, we have

\begin{align*}
 \int_{\hn} |\Delta_{\hn} u|^2 \ dv_{\hn} & =
\sum_{n=0}^{\infty} \int_{0}^{\infty} \left( d_{n}^{\prime \prime}(r) - \frac{(N-1)(N-3)}{4} \coth^2 r d_{n}(r) \right. \notag \\
&\left. - \frac{(N-1)}{2} d_{n}(r) - \frac{\lambda_{n}}{\sinh^2 r} d_{n}(r) \right)^2 \ dr,
\end{align*}
where the eigenvalues $\lambda_{n}$ are repeated according to their multiplicity.  We consider any given term in the series above and write it as follows:
\begin{equation}
\begin{aligned}\label{rellichpf1}
& \int_{0}^{\infty} \left( d_{n}^{\prime \prime}(r) - \frac{(N-1)(N-3)}{4} \coth^2 r d_{n}(r)
 - \frac{(N-1)}{2} d_{n}(r) - \frac{\lambda_{n}}{\sinh^2 r} d_{n}(r) \right)^2 \ dr  \\
&=  \int_{0}^{\infty} (d_{n}^{\prime \prime}(r))^2 \ dr + \int_{0}^{\infty} \left( \frac{(N-1)(N-3)}{4} \coth^2 r + \frac{(N-1)}{2}
+ \frac{\lambda_{n}}{\sinh^2 r} \right)^2 d_{n}^2 \ dr \\
&- \left( \frac{(N-1)(N-3)}{2} \coth^2 r + (N-1) + \frac{2 \lambda_{n}}{\sinh^2 r } \right) d_{n}^{\prime \prime}(r) d_{n}(r) \ dr \\
& = \int_{0}^{\infty} (d_{n}^{\prime \prime}(r))^2 \ dr + \int_{0}^{\infty} \left( \frac{(N-1)^2}{4} d_{n}^2(r)
+ \frac{\lambda_{n}^2}{\sinh^4 r} d_{n}^2(r) \right.    \\
&\left. + \left( \frac{(N-1)(N-3)}{4} \right)^2 \coth^4 r d_{n}^2(r) + \frac{(N-1)^2 (N-3)}{4} \coth^2 r d_{n}^2 (r)
+ \frac{(N-1) \lambda_{n}}{\sinh^2 r} d_{n}^{2} (r) \right.    \\
&+ \left. \frac{(N-1)(N-3) \lambda_{n}}{2} \frac{\coth^2r}{\sinh^2 r} d_{n}^2 \right) \ dr
- \frac{(N-1)(N-3)}{2} \int_{0}^{\infty} \coth^2 r d_{n}^{\prime \prime}(r) d_{n}(r) \ dr   \\
 & -(N-1) \int_{0}^{\infty} d_{n}^{\prime \prime}(r) d_{n}(r) \ dr
 - 2 \lambda_{n} \int_{0}^{\infty} \frac{1}{\sinh^2 r} d_{n}^{\prime \prime}(r) d_{n}(r) \ dr.   \\
\end{aligned}\end{equation}
 Now we consider each term separately. First let us evaluate the negative terms using integration by parts :
\begin{equation}
 \begin{aligned}\label{rellichpf2}
 &\frac{(N-1)(N-3)}{2}  \int_{0}^{\infty} \coth^2 r d_{n}^{\prime \prime}(r) d_{n}(r) \ dr =  - \frac{(N-1)(N-3)}{2} \int_{0}^{\infty}   (d_{n}^{\prime}(r))^2 \ dr  \\
& - \frac{(N-1)(N-3)}{2} \int_{0}^{\infty}  \frac{1}{\sinh^2 r}   (d_{n}^{\prime}(r))^2 \ dr + \frac{(N-1)(N-3)}{2} \int_{0}^{\infty} \frac{\coth r}{\sinh^2 r} \frac{d}{dr} (d_{n}(r))^2 \ dr \\
& = - \frac{(N-1)(N-3)}{2} \int_{0}^{\infty}   (d_{n}^{\prime}(r))^2 \ dr  - \frac{(N-1)(N-3)}{2} \int_{0}^{\infty}  \frac{1}{\sinh^2 r}   (d_{n}^{\prime}(r))^2 \ dr \\
& + \frac{3}{2} (N-1)(N-3) \int_{0}^{\infty} \frac{1}{\sinh^4 r } (d_{n}(r))^2 \ dr + (N-1)(N-3) \int_{0}^{\infty} \frac{1}{\sinh^2 r } (d_{n}(r))^2 \ dr.
 \end{aligned}
 \end{equation}

\begin{align}\label{rellichpf3}
(N-1) \int_{0}^{\infty} d_{n}^{\prime \prime}(r) d_{n}(r) \ dr =  -(N-1) \int_{0}^{\infty} (d_{n}^{\prime}(r))^2 \ dr.
\end{align}

\begin{align}\label{rellichpf4}
2 \lambda_{n}  \int_{0}^{\infty} \frac{1}{\sinh^2 r}  d_{n}^{\prime \prime}(r) d_{n}(r) \ dr & = - 2 \lambda_{n} \int_{0}^{\infty} \frac{1}{\sinh^2 r} (d_{n}^{\prime}(r))^2 \ dr
+ 2 \lambda_{n} \int_{0}^{\infty} \frac{\coth r}{\sinh^2 r} \frac{d}{dr} (d_{n}(r))^2 \ dr \notag \\
& = - 2 \lambda_{n} \int_{0}^{\infty} \frac{1}{\sinh^2 r} (d_{n}^{\prime}(r))^2 \ dr  + 6 \lambda_{n} \int_{0}^{\infty} \frac{1}{\sinh^4 r} (d_{n}(r))^2 \ dr \notag \\
& + 4 \lambda_{n} \int_{0}^{\infty} \frac{1}{\sinh^2 r} (d_{n}(r))^2 \ dr.
\end{align}
 Taking in to account \eqref{rellichpf2}, \eqref{rellichpf3}, \eqref{rellichpf4} and inserting in \eqref{rellichpf1} we get,

 \begin{align*}
&\int_{0}^{\infty} \left( d_{n}^{\prime \prime}(r) - \frac{(N-1)(N-3)}{4} \coth^2 r d_{n}(r)
 - \frac{(N-1)}{2} d_{n}(r) - \frac{\lambda_{n}}{\sinh^2 r} d_{n}(r) \right)^2 \ dr \\
&= \int_{0}^{\infty} (d_{n}^{\prime \prime}(r))^2 \ dr + (N-1) \int_{0}^{\infty} (d_{n}^{\prime}(r))^2 \ dr
+ \frac{(N-1)(N-3)}{2} \int_{0}^{\infty}  (d_{n}^{\prime} (r))^2 \ dr \\
& + \frac{(N-1)(N-3)}{2} \int_{0}^{\infty}    \frac{1}{\sinh^2 r} (d_{n}^{\prime}(r))^2 \ dr +
2 \lambda_{n} \int_{0}^{\infty} \frac{1}{\sinh^2 r} (d_{n}^{\prime}(r)^2 \ dr \\
&+ \int_{0}^{\infty} \left( \frac{(N-1)^2}{4} (d_{n}(r))^2
+ \frac{\lambda_{n}^2}{\sinh^4 r} (d_{n}(r))^2
 + \left( \frac{(N-1)(N-3)}{4} \right)^2 \coth^4 r (d_{n}(r))^2 \right. \\
& \left.  + \frac{(N-1)^2 (N-3)}{4} \coth^2 r (d_{n}(r))^2   + \frac{(N-1) \lambda_{n}}{\sinh^2 r} (d_{n} (r))^2
 +  \frac{(N-1)(N-3) \lambda_{n}}{2} \frac{\coth^2r}{\sinh^2 r} (d_{n}(r))^2 \right) \ dr  \\
& -   \frac{3}{2}(N-1)(N-3) \int_{0}^{\infty} \frac{1}{\sinh^4 r} (d_{n}(r))^2  \ dr
 - (N-1)(N-3) \int_{0}^{\infty} \frac{1}{\sinh^2 r} (d_{n} (r))^2 \ dr \\
&  - 6 \lambda_{n} \int_{0}^{\infty} \frac{d_{n}^2(r)}{\sinh^4 r} \ dr  - 4 \lambda_{n} \int_{0}^{\infty} \frac{d_{n}^2(r)}{\sinh^2 r} \ dr.
\end{align*}

Upon simplifying further we get,

\begin{align*}
&\int_{0}^{\infty} \left( d_{n}^{\prime \prime}(r) - \frac{(N-1)(N-3)}{4} \coth^2 r d_{n}(r)
 - \frac{(N-1)}{2} d_{n}(r) - \frac{\lambda_{n}}{\sinh^2 r} d_{n}(r) \right)^2 \ dr \\ & = \int_{0}^{\infty} (d_{n}^{\prime \prime}(r))^2 \ dr + \frac{(N-1)^2}{2}  \int_{0}^{\infty} (d_{n}^{\prime}(r))^2 \ dr  + \left( \frac{(N-1)(N-3)}{2} +
2 \lambda_{n} \right)  \int_{0}^{\infty} \frac{1}{\sinh^ 2 r}  (d_{n}^{\prime}(r))^2 \ dr   \notag \\
& \frac{(N-1)^4}{16} \int_{0}^{\infty} (d_{n}(r))^2 \ dr + \left( \lambda_{n}^2 + \frac{(N-1)(N-3)}{2} \lambda_{n} - 6 \lambda_{n} + \frac{(N-1)^2 (N-3)^2}{16} \right. \notag \\
& \left. - \frac{3}{2} (N-1)(N-3) \right) \int_{0}^{\infty} \frac{1}{\sinh^4 r} (d_{n}(r))^2 \ dr  + \left( \frac{(N-1)^2(N-3)^2}{8} + \frac{(N-1)^2(N-3)}{4}  \right. \notag \\
& \left. + \frac{(N-1)(N-3)}{2} \lambda_{n} + (N-5) \lambda_{n} - (N-1)(N-3) \right) \int_{0}^{\infty} \frac{1}{\sinh^2 r} (d_{n}(r))^2 \ dr. \notag \\
\notag \end{align*}

In order to estimate the second order term we use the 1-dimensional Rellich inequality \cite{R}:

\[
 \int_{0}^{\infty} d_{n}^{\prime \prime}(r) \ dr \geq \frac{9}{16} \int_{0}^{\infty} \frac{d_{n}^2 (r)}{r^4} \ dr ,
\]
combining this with the Lemma \ref{hardytype}  and one dimensional Hardy inequality we get

\begin{align*}
 & \int_{0}^{\infty} \left( d_{n}^{\prime \prime}(r) - \frac{(N-1)(N-3)}{4} \coth^2 r d_{n}(r)
 - \frac{(N-1)}{2} d_{n}(r) - \frac{\lambda_{n}}{\sinh^2 r} d_{n}(r) \right)^2 \ dr \\
& \geq \frac{9}{16} \int_{0}^{\infty} \frac{d_{n}^2(r)}{r^4} \ dr + \frac{(N-1)^2}{8} \int_{0}^{\infty} \frac{d_{n}^2(r)}{r^2} \ dr
 + \frac{(N-1)^4}{16} \int_{0}^{\infty} d_{n}^{2}(r) \ dr \\
& + A_{n} \int_{0}^{\infty} \frac{d_{n}^2(r)}{\sinh^4 r} \ dr + B_{n} \int_{0}^{\infty} \frac{d_{n}^2 (r)}{\sinh^2 r} \ dr,
\end{align*}
where
\[
 A_{n} = \left[ \lambda_{n}^2 +  \frac{N(N-4)}{2} \lambda_{n} +
\frac{((N-1)(N-3))^2}{16} - \frac{3}{8} (N-1)(N-3) \right]
\]
and
\[
B_{n} = \left[   \frac{(N+1)(N-3)}{2} \lambda_{n} + \frac{(N-1)^2(N-3)}{4} + \frac{((N-1)(N-3))^2}{8} - \frac{(N-1)(N-3)}{2} \right].
\]
We note that
\[
\min_{n \in \mathbb{N}_{0}} A_{n}=\frac{(N-1)(N-3)(N^2-4N-3)}{16} \quad \mbox{and} \quad \min_{n \in \mathbb{N}_{0}} B_{n}=\frac{(N^2-1)(N-3)^2}{8}
\]
so that they are both positive for $N\geq 5$. Also we have

\[
 \int_{\hn} u^2 \ dv_{\hn} = \int_{\hn} v^2 (\sinh r)^{-(N-1)} \ dv_{\hn} = \sum_{n= 0}^{\infty} \int_{0}^{\infty} d_{n}^2 \ dr,
\]
similarly,
\[
 \int_{\hn} \frac{u^2}{r^2} \ dv_{\hn} = \sum_{n=0}^{\infty} \int_{0}^{\infty} \frac{d_{n}^2(r)}{r^2} \ dr,
\]
and so on.

Now using all these facts we obtain
\begin{align*}
 \int_{\hn} |\Delta_{\hn} u|^2 \ dv_{\hn} - \frac{(N-1)^4}{16} \int_{\hn} u^2 \ dv_{\hn}
\geq \frac{9}{16} \int_{\hn} \frac{u^2}{r^4} \ dv_{\hn} + \frac{(N-1)^2}{8} \int_{\hn} \frac{u^2}{r^2} \ dv_{\hn}\\
\frac{(N-1)(N-3)(N^2-4N-3)}{16} \int_{\hn} \frac{u^2}{\sinh^4 r} \ dv_{\hn} +\frac{(N^2-1)(N-3)^2}{8} \int_{\hn} \frac{u^2}{\sinh^2 r} \ dv_{\hn},
\end{align*}
namely \eqref{PR}. $\Box$

\subsection{Optimal constant in \eqref{PR}}\label{optimal}
In this section we show the optimality of the first constant in \eqref{PR}. Inspired by \cite{vaz}, we introduce the
following change of variables:
\begin{equation}\label{be1}
 \frac{ds}{s^{N-1}} = \frac{dr}{(\sinh r)^{N-1}}.
\end{equation}

By \eqref{be1}  and restricting to radial functions, one has

\begin{equation*}
 \Delta_{\hn} U(r) = \left( \frac{s}{\sinh r(s)} \right)^{2(N-1)} \Delta V(s),
\end{equation*}

where $V(s) = U(r(s))$ and $\Delta$ denotes the Euclidean Laplacian. Reading \eqref{PR} with the above transformation we have

\begin{prop}\label{onedimensional}
Let $ N \geq 5$ and $r=r(s)$ be as defined in \eqref{be1}. For every $v \in C^{\infty}_c(0,+\infty)$ there holds
\begin{align}\label{weighted}
 \int_{0}^{\infty} \frac{1}{\rho(s)} (\Delta v)^2 s^{N-1} \ ds & \geq \frac{(N-1)^4}{16} \int_{0}^{\infty} \rho(s) v^2 s^{N-1} \ ds\\\notag
& +\frac{9}{16} \int_{0}^{\infty} \frac{\rho(s)}{r^{4}(s)} v^2 s^{N-1} \ ds +
\frac{(N-1)^2}{8} \int_{0}^{\infty} \frac{\rho(s)}{r^{2}(s)} v^2 s^{N-1} \ ds,
\end{align}
  where $\rho(s) = \left(\frac{\sinh r(s)}{s} \right)^{2(N-1)}.$
\end{prop}

\begin{rem}\label{conj}
As the asymptotics performed here below reveal, when $v$ is supported in the complement of a large ball all the constants in \eqref{weighted} coincide with those of the optimal inequality obtained in \cite[Theorem 5.1-(iii)]{CM}. This observation suggests that also the constants found in \eqref{weighted} should be optimal. This cannot, however, be deduced from \cite{CM} since the weight $\rho$ is close to the homogeneous one considered in \cite{CM} only at infinity.
\end{rem}

\begin{proof}
From \eqref{be1} we have
\begin{equation*}\begin{aligned}
 \int_{\hn}{(\Delta_{\hn} u)^2} (\sinh r)^{N-1} \ dr & = \int_{0}^{\infty}
\frac{s^{4(N-1)}}{(\sinh r(s))^{4(N-1)}} (\Delta v)^2 (\sinh r(s))^{2(N-1)} \frac{1}{s^{N-1}} \ ds \notag \\
& = \int_{0}^{\infty} \frac{1}{\rho(s)} (\Delta v)^2 s^{N-1} \ ds.
\end{aligned}
\end{equation*}
\begin{align*}
 \int_{\hn} u^2 (\sinh r)^{N-1} \ dr & = \int_{0}^{\infty} u^2(r(s)) \frac{(\sinh r(s))^{2(N-1)}}{s^{2(N-1)}}
\frac{s^{(N-1)}}{(\sinh r(s))^{N-1}} s^{N-1} \ ds \notag \\
& = \int_{0}^{\infty} v^{2}(s) \rho(s) s^{N-1} \ ds.
\end{align*}
The other integrals in \eqref{PR} can be rewritten similarly. All these terms replaced in \eqref{PR} yields the thesis.

\end{proof}
Next we provide the asympotics of the transformation \eqref{be1}.

\begin{lem}\label{lems}
Let $s=s(r)$ be as given by \eqref{be1}. Then
\begin{equation*}
s(r) = c_{1} e^{\frac{N-1}{N-2} r} - c_{2} e^{-\frac{N-3}{N-2} r} + o(e^{-\frac{N-3}{N-2}r})\quad
\ \mbox{as} \ r \rightarrow \infty,
\end{equation*}
where $c_{1}$ and $c_{2}$ are positive constants.
\end{lem}

\begin{proof}
 From the trasformation \eqref{be1}, we have
\begin{align*}
 s(r) = \frac{(N-2)^{-\frac{1}{N-2}}}{2^{\frac{N-1}{N-2}}} \left( \int_{r}^{\infty}
(e^\sigma - e^{-\sigma})^{-N + 1} \ d\sigma \right)^{-\frac{1}{N-2}}.
\end{align*}
Notice that, as $r \rightarrow \infty,$

\begin{align*}
 \int_{r}^{\infty} (e^{2 \sigma} - 1)^{-N + 1} & e^{\sigma(N-1)} \ d\sigma = \int^{e^{-r}}_{0} (1- y^2)^{-N + 1} y^{N-2} \ dy \\
& = \int_{0}^{e^{-r}} \left[ y^{N-2} + (N-1) y^{N}+ o(y^{N}) \right] \ dy\\
&= \frac{e^{-r(N-1)}}{(N-1)} + \frac{(N-1)}{(N+1)} e^{-r(N+1)}  + o(e^{-r(N+1)}).
\end{align*}
Hence, as $r \rightarrow \infty,$

\begin{align*}
 s(r) & = \frac{(N-2)^{-\frac{1}{N-2}}}{2^{\frac{N-1}{N-2}}} \left[ \frac{e^{-r(N-1)}}{(N-1)} + \frac{N-1}{N+1} e^{-r(N+1)} + o(e^{-r(N+1)}) \right]^{\frac{-1}{N-2}} \\
& = c_1 e^{r\frac{N-1}{N-2}} \left[ 1 -
\frac{(N-1)^2}{(N+1)(N-2)} e^{- 2r} + o(e^{-2r}) \right]\\
& = c_1 e^{r\frac{N-1}{N-2}}
\left[ 1 - \frac{(N-1)^2}{(N+1)(N-2)} e^{-2r} + o(e^{-2r}) \right],
\end{align*}
where $$c_1:= \left(\frac{N-1}{2^{N-1}(N-2)}\right)^{\frac{1}{N-2}}$$

 This proves the lemma setting
$$c_{2} := \frac{(N-1)^2\,c_1}{(N+1)(N-2)}.$$
\end{proof}

Next we need the precise asymptotics of $\rho.$

\begin{lem}\label{lemrho}
 Let $\rho$ be defined as in Proposition \ref{onedimensional}, then

\begin{equation*}
 \rho(s) := \left( \frac{\sinh r(s)}{s} \right)^{2(N-1)} = \left(2c_{1}\right)^{-2N + 2} e^{-2\frac{N-1}{N-2}r(s)} \left( 1 + k_{1} e^{-2r(s)} + o(e^{-2r(s)} \right), \ \mbox{as} \ s \rightarrow \infty,
\end{equation*}
 where $k_{1} = \frac{2(N-1)(c_{2} - c_{1})}{c_{1}}$ and $c_{1}, c_{2}$ are as in the previous lemma.
\end{lem}

\begin{proof}
By Lemma \ref{lems},
 \begin{align*}
  \rho(s) & = \frac{e^{(2N-2)r(s)}}{2^{2N-2}}\left[ \left( 1 - (2N-2)e^{-2r(s)} + o(e^{-2r(s)}) \right) \right. \\
& \left. \left( c_{1} e^{\frac{N-1}{N-2} r(s)} - c_{2} e^{-\frac{N-3}{N-2}r(s)} +
o(e^{-\frac{N-3}{N-2}r(s)}) \right)^{-2N + 2} \right]\\
& = \frac{c_{1}^{-2N + 2}}{2^{2N-2}} e^{-2\frac{N-1}{N-2}r(s)}
 \left( 1 + k_{1} e^{-2r(s)} + o(e^{-2r(s)} \right),
 \end{align*}
this proving the claim.
\end{proof}

Now, following the idea in the proof of \cite[Theorem 5.5]{CM} we can state
\begin{prop}
 If $A \in \mathbb{R}$ is such that
\begin{align}\label{onedimensional1}
 \int_{0}^{\infty} \frac{1}{\rho(s)} (\Delta v)^2 s^{N-1} \ ds &
\geq \frac{(N-1)^4}{16} \int_{0}^{\infty} \rho(s) v^2 s^{N-1} \ ds +A \int_{0}^{\infty} \frac{\rho(s)}{r^{2}(s)} v^2 s^{N-1} \ ds
\end{align}
for every $v \in C_c^{\infty}(0, \infty),$ then $A \leq \frac{(N-1)^2}{8}.$

\end{prop}
\begin{proof}
 Set

\[
 v(s) = s^{-\frac{(N-2)}{2}}w(- \log s),
\]
where $v \in C_c^{\infty}(0, \infty) \Leftrightarrow w \in C_c^{\infty}(-\infty, \infty).$ We compute

\begin{equation}\label{be7}
\Delta v(s) = s^{-\frac{(N+2)}{2}} \left[ w^{\prime \prime}(- \log s) - \frac{(N-2)^2}{4} w(- \log s) \right].
\end{equation}
Inserting \eqref{be7} in \eqref{onedimensional1}, we get

\begin{equation*}
\begin{aligned}
& \int_{0}^{\infty} \frac{s^{-3}}{\rho(s)}  \left[ (w^{{\prime \prime}}(-\log s))^2 +
\left( \frac{N-2}{2} \right)^4 w^2(- \log s) -
2 \left( \frac{N-2}{2} \right)^2 w^{\prime \prime}(-\log s) w(- \log s) \right]  ds \\
& \geq \frac{(N-1)^4}{16} \int_{0}^{\infty} \rho(s) s\, w^2(-\log s) \, ds
+ A \int_{0}^{\infty} \frac{\rho(s)\,s}{r^{2}(s)} \,w({-\log s})  \ ds.\\
\end{aligned}
\end{equation*}
Now substituting $\sigma = - \log s$ we have

\begin{align*}
 & \int_{-\infty}^{\infty} (w^{\prime \prime}(\sigma))^2 \frac{1}{\rho(e^{-\sigma})} e^{2 \sigma} \ d\sigma
+ \left(\frac{N-2}{2}\right)^4 \int_{-\infty}^{\infty} w^2(\sigma) \frac{1}{\rho(e^{-\sigma})} e^{2 \sigma} \ d\sigma \\
&- 2 \left(\frac{N-2}{2} \right)^2 \int_{-\infty}^{\infty} w^{\prime \prime}(\sigma) w(\sigma)
\frac{1}{\rho(e^{-\sigma})} e^{2 \sigma} \ d\sigma \\
& \geq \left( \frac{N-1}{2} \right)^4 \int_{-\infty}^{\infty} w^2(\sigma) \rho(e^{-\sigma}) e^{-2\sigma} \ d \sigma + A \int_{-\infty}^{\infty} w^{2}(\sigma) \frac{\rho(e^{-\sigma})}{r^{2}(e^{-\sigma})} e^{-2\sigma} \ d\sigma,
\end{align*}
for all $w \in C_{c}^{\infty}(-\infty, \infty).$ As above inequality holds true also for $w(\sigma) = z(t\sigma),$ for all
$t > 0$ and $z \in C_{c}^{\infty}(-\infty, 0),$ we obtain

\begin{align*}
& \int_{-\infty}^{0} \left[ t^4 (z^{\prime \prime}(x))^2 + \left( \frac{N-2}{2} \right)^4 z^2(x) -
2 \left( \frac{N-2}{2} \right)^2 t^2 z^{\prime \prime}(x) z(x) \right]
\frac{1}{\rho(e^{\frac{|x|}{t}})\,e^{2 \frac{|x|}{t}}} \ dx \\
& \geq \left( \frac{N-1}{2} \right)^4 \int_{-\infty}^{0} \rho(e^{\frac{|x|}{t}})e^{2 \frac{|x|}{t}} z^{2}(x) \ dx
 + A \int_{-\infty}^{0} \rho(e^{\frac{|x|}{t}})\,e^{2 \frac{|x|}{t}} \frac{z^{2}(x)}{r^{2}(e^{ \frac{|x|}{t}})} \ dx.
\end{align*}
If $x \in K \subset (-\infty, 0),$ where $K$ is a compact set, then $\frac{|x|}{t} \rightarrow \infty$ uniformly as $t \rightarrow 0$.
Since, by Lemma \ref{lems} and Lemma \ref{lemrho}, as $s \rightarrow \infty$
\[
 r(s) = \log \left( \frac{s}{c_{1}} \right)^{\frac{N-2}{N-1}} + o(1)\,,
\]
and
$$\rho(s) s^2=2^{2-2N}c_1^{4-2N}\left(1-\frac{4(N-1)}{N+1}\,e^{-2r(s)}+o(e^{-2r(s)}) \right)\,,$$
the above inequality yields

\begin{align*}
 & t^4\,\frac{c_{1}^{2N - 4}}{2^{2- 2N}}\,\int_{- \infty}^{0} (z^{\prime \prime})^2 \,
\left( 1 + \frac{4(N-1)}{N+1} \left( \frac{e^{-2\frac{|x|}{t}}}{c_{1}}
\right)^{\frac{N-2}{N-1}} + o(e^{\frac{-2(N-2)}{N-1}\frac{|x|}{t}}) \right)  \ dx  \\
& - t^2 \left[ \frac{(N-1)^2}{2}\int_{-\infty}^{0} z^{\prime \prime} z \left( 1 +  \frac{4(N-1)}{N+1} \left( \frac{e^{-2\frac{|x|}{t}}}{c_{1}}
\right)^{\frac{N-2}{N-1}} +
o(e^{\frac{-2(N-2)}{N-1}\frac{|x|}{t}}) \right)\,dx \right. \\
& \left. + A \int_{-\infty}^{0} \frac{z^2}{|x|^2} \left( 1 - \frac{4(N-1)}{N+1}  \left( \frac{e^{-2\frac{|x|}{t}}}{c_{1}}
\right)^{\frac{N-2}{N-1}} +
o(e^{\frac{-2(N-2)}{N-1}\frac{|x|}{t}}) \right) \ dx \right] \\
& + \frac{(N-2)^2 (N-1)^2}{2^4}\left[ \int_{-\infty}^{0} z^2 \left(  \frac{8(N-1)}{N+1}  \left( \frac{e^{-2\frac{|x|}{t}}}{c_{1}}
\right)^{\frac{N-2}{N-1}}+
o(e^{\frac{-2(N-2)}{N-1}\frac{|x|}{t}}) \right) \ dx  \right]  \geq 0\,.
\end{align*}
Hence, as $t \rightarrow 0$  and
integrating by parts, we obtain

\begin{align*}
 \frac{(N-1)^2}{2} \int_{-\infty}^{0} (z^{\prime}(x))^2 \ dx \geq A \int_{-\infty}^{0} \frac{z^2}{|x|^2} \ dx.
\end{align*}
Hence,
\[
 \frac{A}{\frac{(N-1)^2}{2}} \leq \inf_{v \in C_c^{\infty}(-\infty, 0)}
\frac{\int_{-\infty}^{0} |z^{\prime}|^2}{\int_{-\infty}^{0} |x|^{-2} |z|^2} = \frac{1}{4}
\]
and we conclude.

\end{proof}

\section{Proof of Corollary \ref{corHALF1} and Corollary \ref{cor2}}

\emph{Proof of Corollary \ref{corHALF1}.}\par
\smallskip\par
From the transformation \eqref{halftransformation} and since $r=d((x,y),(0,1))$, we obtain that

\[
\int_{\mathbb{H}^{N}} u^2 \ dv_{\mathbb{H}^{N}} = \int_{\mathbb{R}^{+}} \int_{\mathbb{R}^{N-1}} \frac{v^2}{y^2} \ dx \ dy \,,\quad \int_{\mathbb{H}^{N}} \frac{u^2}{r^2} \ dv_{\mathbb{H}^{N}} = \int_{\mathbb{R}^{+}} \int_{\mathbb{R}^{N-1}} \frac{v^2}{y^2 d^2} \ dx \ dy
\]
and
\[
\int_{\mathbb{H}^{N}} |\nabla_{\mathbb{H}^{N}} u|^2 \ dv_{\mathbb{H}^{N}} =  \int_{\mathbb{R}^{+}} \int_{\mathbb{R}^{N-1}} |\nabla v|^2 \ dx \ dy +
\left( \frac{(N-1)^2}{4} - \frac{1}{4} \right)   \int_{\mathbb{R}^{+}} \int_{\mathbb{R}^{N-1}} \frac{v^2}{y^2} \ dx \ dy
\]

Inserting the above identities into \eqref{poincare}, \eqref{mazya1} follows at once together with the optimality of the constants that comes from those in \eqref{poincare}.\qed
\par
\medskip\par
\emph{Proof of Corollary \ref{mazya}}
\par
\smallskip\par

For all $u \in C_c^{\infty}(\mathbb{H}^{N}),$ we replace transformation \eqref{halftransformation} with
$$v(x,y):= |y|^{-\frac{N-2}{2}} u(x,|y|), \ \ x \in \mathbb{R}^{N-1}, y \in \mathbb{R}^{k}\,.$$
Hence, $v\in C_c^{\infty}( \mathbb{R}^{N+k-1})$ has cylindrical symmetry and compact support in $\mathbb{R}^{N+k-1} \setminus \mathbb{R}^{N-1}$. Then, the same density argument of \cite[Appendix B]{mancini} allows to conclude that
\[
\omega_k\int_{\mathbb{H}^{N}} u^2 \ dv_{\mathbb{H}^{N}} = \int_{\mathbb{R}^{k}} \int_{\mathbb{R}^{N-1}} \frac{v^2}{y^2} \ dx \ dy \,,\quad \omega_k\int_{\mathbb{H}^{N}} \frac{u^2}{r^2} \ dv_{\mathbb{H}^{N}} = \int_{\mathbb{R}^{k}} \int_{\mathbb{R}^{N-1}} \frac{v^2}{y^2 d^2} \ dx \ dy
\]
and
\[
\omega_k \int_{\mathbb{H}^{N}} |\nabla_{\mathbb{H}^{N}} u|^2 \ dv_{\mathbb{H}^{N}} =  \int_{\mathbb{R}^{k}} \int_{\mathbb{R}^{N-1}} |\nabla v|^2 \ dx \ dy +
\left( \frac{(N-1)^2}{4} - \frac{1}{4} \right)  \int_{\mathbb{R}^{k}} \int_{\mathbb{R}^{N-1}} \frac{v^2}{y^2} \ dx \ dy \,,
\]
for all $v=v(x,y) \in C_c^{\infty}(\mathbb{R}^{N-1} \times \mathbb{R}^{k}),$ with $v(x,0)=0$ if $k=1$, where $\omega_k$ is the volume of the $k$ dimensional unit sphere. Using the above identities in \eqref{poincare} yields the claim.

\qed

\par
\medskip\par
\emph{Proof of Corollary \ref{cor2}}
\par
\smallskip\par
Before going further we state the following

\begin{lem}\label{Laconformal}
Let $ N \geq 3.$ For $u \in C_c^{\infty}(\mathbb{H}^{N}),$ define $v(x, y) := y^{\alpha} u(x,y), \ (x,y) \in \mathbb{R}^{N-1} \times \mathbb{R}^{+},$ then
\begin{equation*}
\Delta_{\mathbb{H}^{N}} u = y^{\alpha + 2} \Delta v + (2 \alpha - (N-2)) y^{\alpha} \frac{\partial v}{\partial y}  + \alpha (\alpha - (N-1)) y^{\alpha} v\,.
\end{equation*}

\end{lem}
\begin{proof}
The proof follows by considering hyperbolic space as upper half space model $\mathbb{R}^{N}_{+} = \{ (x, y) \in \mathbb{R}^{N-1} \times \mathbb{R}^{+} \} $ endowed with the Riemannian metric $\frac{\delta_{ij}}{y^2}$ and using the explicit expression of Laplacian in these coordinates, namely
$\Delta_{\mathbb{H}^{N}} = y^2 \Delta - (N-2) y \frac{\partial}{\partial y}  $.
\end{proof}
\emph{Proof of \eqref{HALFrellich1}.} Let $u \in C_c^{\infty}(\mathbb{H}^{N})$, from the transformation \eqref{higherorder} and Lemma \ref{Laconformal} with $\alpha=(N-2)/2$, we deduce that
$$
\int_{\mathbb{H}^{N}} u^2 \ dv_{\mathbb{H}^{N}} = \int_{\mathbb{R}^{+}} \int_{\mathbb{R}^{N-1}} \frac{v^2}{y^2} \ dx \ dy \,,
$$
$$\int_{\mathbb{H}^{N}} \frac{u^2}{r^2} \ dv_{\mathbb{H}^{N}} = \int_{\mathbb{R}^{+}} \int_{\mathbb{R}^{N-1}} \frac{v^2}{y^2 d^2} \ dx \ dy\,,\quad \int_{\mathbb{H}^{N}} \frac{u^2}{r^4} \ dv_{\mathbb{H}^{N}} = \int_{\mathbb{R}^{+}} \int_{\mathbb{R}^{N-1}} \frac{v^2}{y^2 d^4} \ dx \ dy
$$
and
\begin{align*}\int_{\mathbb{H}^{N}} |\Delta_{\mathbb{H}^{N}} u|^2 \ dv_{\mathbb{H}^{N}} &=  \int_{\mathbb{R}^{+}} \int_{\mathbb{R}^{N-1}} y^2|\Delta v|^2 \ dx \ dy + \frac{N(N-2)}{2}  \int_{\mathbb{R}^{+}} \int_{\mathbb{R}^{N-1}} |\nabla v|^2 \ dx \ dy\\&+
\frac{N^2(N-2)^2}{16}   \int_{\mathbb{R}^{+}} \int_{\mathbb{R}^{N-1}} \frac{v^2}{y^2} \ dx \ dy
\end{align*}
where $r=d((x,y),(0,1))$. The above identities inserted into \eqref{PR} yields \eqref{HALFrellich1}.
Next we turn to the optimality issues. Assume by contradiction that  the following inequality holds $$\int_{\mathbb{R}^{+}} \int_{\mathbb{R}^{N-1}} \left( y^2 (\Delta v)^2 +
c |\nabla v|^2 \right) \ dx \ dy
 \geq \frac{2N^2 - 4N +1}{16} \int_{\mathbb{R}^{+}} \int_{\mathbb{R}^{N-1}} \frac{v^2}{y^2} \ dx \ dy $$
for all $ u \in C^{\infty}_{c}(\mathbb{H}^{N})$ with $c< \frac{N(N-2)}{2}$. The above inequality, jointly with \eqref{halftransformation} and \eqref{mazya1}, yields
$$\int_{\mathbb{H}^{N}} |\Delta_{\mathbb{H}^{N}} u|^2 \ dv_{\mathbb{H}^{N}}\geq\frac{(N -1)^2}{16} \int_{\mathbb{H}^{N}} u^2 \ dv_{\mathbb{H}^{N}}+\left(\frac{N(N-2)}{2}-c\right)  \int_{\mathbb{R}^{+}} \int_{\mathbb{R}^{N-1}} |\nabla v|^2 \ dx \ dy$$
$$\geq\frac{(N -1)^2}{16} \int_{\mathbb{H}^{N}} u^2 \ dv_{\mathbb{H}^{N}}+\frac{1}{4}  \left(\frac{N(N-2)}{2}-c\right)  \int_{\mathbb{R}^{+}} \int_{\mathbb{R}^{N-1}} \frac{v^2}{y^2} \ dx \ dy$$
$$=\left(\frac{(N -1)^2}{16}+\frac{1}{4}  \left(\frac{N(N-2)}{2}-c\right) \right) \int_{\mathbb{H}^{N}} u^2 \ dv_{\mathbb{H}^{N}}\,,$$
a contradiction with \eqref{poin}. The optimality of the other constants follows straightforwardly from what remarked above. \par
\emph{Proof of \eqref{HALFrellich2}.} Let $u \in C_c^{\infty}(\mathbb{H}^{N}).$ One could proceed by using a change of variable in \eqref{HALFrellich1}, but we give a short proof using the method just used above. By exploiting the transformation \eqref{higherorder} with $\alpha=(N-4)/4$, we have

\[
\int_{\mathbb{H}^{N}}  u^2 \ dv_{\mathbb{H}^{N}} = \int_{\mathbb{R}^{+}} \int_{\mathbb{R}^{N-1}} \frac{v^2}{y^4} \ dx \ dy\,
\]
and
$$\int_{\mathbb{H}^{N}} \frac{u^2}{r^2} \ dv_{\mathbb{H}^{N}} = \int_{\mathbb{R}^{+}} \int_{\mathbb{R}^{N-1}} \frac{v^2}{y^4 d^2} \ dx \ dy\,,\quad \int_{\mathbb{H}^{N}} \frac{u^2}{r^4} \ dv_{\mathbb{H}^{N}} = \int_{\mathbb{R}^{+}} \int_{\mathbb{R}^{N-1}} \frac{v^2}{y^4 d^4} \ dx \ dy
$$

Furthermore, by Lemma \ref{Laconformal}  $\alpha=(N-4)/2$
and by integration by parts, we obtain

\begin{align}
\int_{\mathbb{H}^{N}} (\Delta_{\mathbb{H}^{N}} u )^2 \ dv_{\mathbb{H}^{N}} & = \int_{\mathbb{R}^{+}} \int_{\mathbb{R}^{N-1}} \left( (\Delta v)^2 + 4 \frac{v_{y}^2}{y^2}
 + \frac{(N-4)^2 (N+2)^2}{16} \frac{v^2}{y^4} \right.  \notag \\
& \left. - 4 \frac{v_{y}}{y} \Delta v - \frac{(N-4)(N+2)}{2} \frac{v \Delta v}{y^2}
  + (N-4)(N+2) \frac{vv_{y}}{y^3} \right) \ dx \ dy \notag \\
 & = \int_{\mathbb{R}^{+}} \int_{\mathbb{R}^{N-1}} \left( (\Delta v)^2 + \frac{(N-4)^2(N+2)^2}{16} \frac{v^2}{y^4} + \frac{(N^2 - 2N - 4)}{2} \frac{|\nabla v|^2}{y^2} \right) \ dx \ dy. \notag
\end{align}

Taking into account the above relations in \eqref{PR} we obtain \eqref{HALFrellich2}. As concerns the optimality of the constants, it follows in the same way of \eqref{HALFrellich1}. The main difference is that here, to show the optimality of the constant in front of the term involving the gradient, \eqref{mazya1} has to be replaced by the inequality
$$\int_{\mathbb{R}^{+}} \int_{\mathbb{R}^{N-1}} \frac{|\nabla v|^2}{y^2} \ dx \ dy \geq \frac{9}{4} \int_{\mathbb{R}^{+}} \int_{\mathbb{R}^{N-1}}\frac{v^2}{y^4}\ dx \ dy\,$$
for $v\in C_c^{\infty}(\mathbb{R}_+^{N})$, the proof of which is readily obtained combining integration by part with H\"older inequality.
 \qed

\section{Proof of Proposition \ref{infinity}}\label{loginfinity}
The proof follows by exploiting several ideas from \cite[Theorem 6.1]{FT} and is divided in three steps. \\

{\bf{Step 1.}} Let us denote $\tilde \Phi_{k} (r) = \Phi(r) f_{k}(r),$ where $\Phi (r) = \left( \frac{r}{\sinh r} \right)^{\frac{N-1}{2}},$ then using Proposition 4.3 we have

\begin{align}\label{prop4.3}
-\Delta_{\hn} \tilde \Phi_{k}(r) - \frac{(N-1)^2}{4} \tilde \Phi_{k}(r) & = \frac{(N-1)(N-3)}{4} \frac{1}{\sinh^2 r} \tilde \Phi_{k}(r) - \frac{(N-1)(N-3)}{4} \frac{1}{r^2} \tilde \Phi_{k}(r) \notag \\
& - (f^{\prime \prime}_{k}(r) - \frac{(N-1)}{r} f_{k}(r)) \Phi.
\end{align}

Set $f_{0}(r) = r^{\frac{2-N}{2}}$ and, for $k = 1, 2, \ldots $,

\begin{equation}\label{fk}
f_{k}(r) = r^{\frac{(2-N)}{2}} X_{1}^{-\frac{1}{2}}(r) X_{2}^{-\frac{1}{2}}(r) \ldots  X_{k}^{-\frac{1}{2}}(r)\,,
\end{equation}
from \cite[Theorem 6.1]{FT} we know that
\[
-  (f^{\prime \prime}_{k}(r) - \frac{(N-1)}{r} f_{k}(r)) = \frac{1}{r^2} \left( \frac{(N-2)^2}{4} + \frac{1}{4}X_{1}^2 +  \frac{1}{4}X_{1}^2 X_{2}^2 + \ldots +
\frac{1}{4} X_{1}^2 \ldots X_{k}^2 \right).
\]
Now substituting \eqref{fk} in \eqref{prop4.3} we obtain,
\begin{align}\label{infiniteequation}
-\Delta_{\hn} \tilde \Phi_{k}(r) - \frac{(N-1)^2}{4} \tilde \Phi_{k}(r) & = \frac{(N-1)(N-3)}{4} \frac{1}{\sinh^2 r} \tilde \Phi_{k}(r) -  \frac{1}{4r^2} \tilde \Phi_{k}(r) \notag \\
& + \frac{1}{4} \sum_{i = 1}^{k} \frac{1}{r^2} X_{1}^2(r)X_{2}^2(r)\ldots X_{i}^2(r) \tilde \Phi_{k}(r)
\end{align}

{\bf{Step 2.}} We consider $u \in C_{c}^{\infty}(B)$ and settled $u(x) = \Psi(x) v(x)$ we compute

\[
\int_{B} |\nabla_{\hn} u|^2 \ dv_{\hn} = \int_{B} \Psi^2 |\nabla_{\hn} v|^2  \ dv_{\hn} + \int_{B} v^2 |\nabla_{\hn} \Psi|^2 \ dv_{\hn}
+ 2 \int_{B}  v \Psi \langle \nabla_{\hn} v\,, \nabla_{\hn} \Psi \rangle \ dv_{\hn},
\]
 after integration by parts the last term of above expression we obtain

 \begin{align}\label{incom}
 \int_{B} |\nabla_{\hn} u|^2 \ dv_{\hn} & = - \int_{B} \Psi (\Delta_{\hn} \Psi) v^2 \ dv_{\hn} + \int_{B} \Psi^2 |\nabla_{\hn} v|^2 \ dv_{\hn} \notag \\
 & = - \int_{B} \frac{\Delta \Psi}{\Psi} u^2 \ dv_{\hn} + \int_{B} \Psi^2 |\nabla_{\hn} v|^2 \ dv_{\hn} \notag \\
& \geq  - \int_{B} \frac{\Delta \Psi}{\Psi} u^2 \ dv_{\hn}.
\end{align}
By choosing $\Psi = \Phi f_{k} $, where $f_{k}$ defined as in \eqref{fk}, in \eqref{incom} and taking the limit $k \rightarrow \infty,$ we derive \eqref{infinityi}. \\

{\bf{Step 3.}}  Next we prove the optimality issue. Let us denote

\begin{align*}
I_{k}(u) & := \int_{B} |\nabla_{\hn} u|^2 \ dv_{\hn} - \frac{(N-1)^2}{4} \int_{B} u^2 - \frac{(N-1)(N-3)}{4} \int_{B} \frac{u^2}{\sinh^2 r} - \frac{1}{4} \int_{B} \frac{u^2}{r^2} \ dv_{\hn} \notag \\
& - \frac{1}{4} \sum_{i = 1}^{k} \int_{B} \frac{u^2}{r^2} X_{1}^2 X_{2}^2 \ldots X_{k}^2 \ dv_{\hn}.
\end{align*}
 Then clearly, for $k = 1, 2 \ldots $
 \[
 I_{k-1}(u) = I_{k}(u) +  \frac{1}{4} \int_{B} \frac{u^2}{r^2} X_{1}^2 X_{2}^2 \ldots X_{k}^2 \ dv_{\hn}.
 \]
 Then it easy to note that

 \begin{equation}\label{optimality1}
 \frac{I_{k-1}(u)}{\int_{B} \frac{u^2}{r^2} X_{1}^2 X_{2}^2 \ldots X_{k}^2 \ dv_{\hn}} = \frac{I_{k}(u)}{\int_{B} \frac{u^2}{r^2} X_{1}^2 X_{2}^2 \ldots X_{k}^2 \ dv_{\hn}} +
  \frac{1}{4}.
 \end{equation}
 By choosing $u = \tilde \Phi_{k} v $ (with $\tilde \Phi_{k}$ as in step 1) and following step 2, we obtain

 \[
 I_{k} (u) = \int_{B} \tilde \Phi_{k}^2 |\nabla_{\hn} v|^2 \ dv_{\hn},
 \]
 and hence
\[
 \frac{I_{k-1}(u)}{\int_{B} \frac{u^2}{r^2} X_{1}^2 X_{2}^2 \ldots X_{k}^2 \ dv_{\hn}}
 = \frac{ \int_{B} \tilde \Phi_{k}^2 |\nabla_{\hn} v|^2 \ dv_{\hn}  }{\int_{B} \frac{u^2}{r^2} X_{1}^2 X_{2}^2 \ldots X_{k}^2 \ dv_{\hn}} +
  \frac{1}{4}.
\]
Now we choose $v=U_{\epsilon, a}\,,$ where
\begin{equation}\label{optimality2}
U_{\epsilon, a}(r) = v_{\epsilon, a}(r) \psi(r) = r^{\epsilon} X_{1}^{a_{1}} X_{2}^{a_{2}} \ldots X_{k}^{a_{k}} \psi(r),
\end{equation}
the parameters $\epsilon, a_{i}$ will be positive and small and eventually will be sent to zero. The function $\psi(r)$ is  a smooth cut-off function such that
$\psi(r) = 1$ in $B_{\delta}$ and $\psi (r) = 0$ outside $B_{2 \delta}$ for some $\delta$ small. \\
Arguing exactly as in \cite[Theorem 6.1]{FT} one sees that as the parameters $\epsilon$ and $a_{i}$ go to zero, then
\[
\frac{ \int_{B} \tilde \Phi_{k}^2 |\nabla_{\hn} v|^2 \ dv_{\hn}  }{\int_{B} \frac{u^2}{r^2} X_{1}^2 X_{2}^2 \ldots X_{k}^2 \ dv_{\hn}} \rightarrow 0.
\]
This immediately gives
\[
 \inf_{C_{c}^{\infty}(B)} \frac{I_{k-1}(u)}{\int_{B} \frac{u^2}{r^2} X_{1}^2 X_{2}^2 \ldots X_{k}^2 \ dv_{\hn}}  \leq \frac{1}{4}
\]
and proves the optimality issue.

\qed

\par\bigskip\noindent
\textbf{Acknowledgments.} We are grateful to Y. Pinchover for explaining to us some of the results and methods of \cite{pinch, pinch2}. We are also grateful to P. Caldiroli for useful discussions. The first and second authors are partially supported by the Research Project FIR (Futuro in Ricerca) 2013 \emph{Geometrical and qualitative aspects of PDE's}. The third author is partially supported by the PRIN project {\em Equazioni alle derivate parziali di tipo ellittico e parabolico: aspetti geometrici, disuguaglianze collegate, e applicazioni}. The first and third authors are members of the Gruppo Nazionale per l'Analisi Matematica, la Probabilit\`a e le loro Applicazioni (GNAMPA) of the Istituto Nazionale di Alta Matematica (INdAM).


\end{document}